\documentclass[11pt]{amsart}

\usepackage{amssymb, amsthm, amsmath}
\usepackage{bbm}
\usepackage{xcolor}
\usepackage{graphicx}
\newtheorem{theorem}{Theorem}[section]
\newtheorem{fact}[theorem]{Fact}
\newtheorem{lemma}[theorem]{Lemma}
\newtheorem{corollary}[theorem]{Corollary}

\newtheorem*{theorem*}{Theorem}{\bf}{\it}
\newtheorem*{proposition*}{Proposition}{\bf}{\it}
\newtheorem*{observation*}{Observation}{\bf}{\it}
\newtheorem*{lemma*}{Lemma}{\bf}{\it}
\newtheorem{claim}[theorem]{Claim}{\bf}{\it}

\theoremstyle{definition}

\theoremstyle{remark}
\newtheorem{remark}[theorem]{Remark}

\def\XXint#1#2#3{{\setbox0=\hbox{$#1{#2#3}{\int}$ }
\vcenter{\hbox{$#2#3$ }}\kern-.6\wd0}}

\begin{document}
\title[]{The Landis conjecture on exponential decay}

\author{A. Logunov}
\address{Alexander Logunov: Department of Mathematics, Princeton University, Princeton, NJ, USA}
\email{log239@yandex.ru}
\author{E. Malinnikova}
\address{Eugenia Malinnikova: Department of Mathematics, Stanford University, Stanford, CA, USA}
\email{eugeniam@stanford.edu}
\author{N. Nadirashvili}
\address{Nicolai Nadirashvili: \phantom{n}CNRS, \phantom{n}Institut de Math\' ematique de Marseille, Marseille, France 
}
\email{nikolay.nadirashvili@univ-amu.fr}
\author{F. Nazarov}
\address{Fedor Nazarov: Department of Mathematics, Kent State University, Kent, OH, USA}
\email{nazarov@math.kent.edu}
\begin{abstract}
Consider a solution $u$ to $\Delta u +Vu=0$ on $\mathbb{R}^2$, where  $V$ is real-valued, measurable and $|V|\leq 1$. 
If $|u(x)| \leq \exp(-C |x| \log^{1/2}|x|)$, $|x|>2$,
where $C$ is a sufficiently large absolute constant, then 
$u\equiv 0$. 
\end{abstract}

\maketitle

 \section{The main result.}

 Let $u$ be a  solution to 
 \begin{equation} \label{eq:schr}
 \Delta u + V u =0
 \end{equation}
  in $\mathbb{R}^n$, where $V$ is
  a measurable function with $|V | \leq 1 $ in the whole space. According to \cite{C15},\cite{KL88}, in the late 1960s Landis conjectured  that if 
  $$ |u(x)| \leq \exp(-C|x|), $$
  where $C>0$ is a sufficiently large constant, then $u\equiv0$. 
 The weaker statement, which was also conjectured by Landis  according to \cite{C15}, states that
 if $|u(x)|$ tends to $0$ faster than exponentially at $\infty$, i.e., $$|u(x)| \leq \exp(-|x|^{1+\varepsilon}),\varepsilon >0,$$ then $u\equiv 0$.
 
 There are two versions of Landis' conjectures: real and complex.
  Meshkov \cite{M92} constructed  a counter-example to the complex version of Landis' conjecture. He showed that there is a complex-valued potential $V$ with $|V|\leq 1$ and a non-zero solution $u$ to \eqref{eq:schr} on $\mathbb{R}^2$ such that $|u(x)| \leq  \exp(-c |x|^{4/3}).$
 Meshkov also showed (in any dimension $n$) that if 
$$ \sup_{\mathbb{R}^n} |u(x)| e^{-\tau |x|^{ 4/3}}< \infty \text{ for all }  \tau>0,$$
then $u\equiv 0$.
The question whether the Landis conjecture is true for real-valued $V$ is open.
The main result of this article confirms the weak version of the Landis conjecture in dimension two.

\begin{theorem} \label{main}
Suppose that $\Delta u +Vu=0$ on $\mathbb{R}^2$, where $u$ and $V$ are real-valued and $|V|\leq 1$. 
If $|u(x)| \leq \exp(-C |x| \log^{1/2}|x|)$, $|x|>2$,
where $C$ is a sufficiently large absolute constant, then 
$u\equiv 0$. 
\end{theorem}

A similar striking difference between the decay estimates for real and complex solutions has also been observed in \cite{C15}, where a closely related equation $\Delta u+W\cdot\nabla u=0$ with a bounded vector field $W:\mathbb{R}^2 \to \mathbb{R}^2$ was studied.

There is a simple example of a solution to \eqref{eq:schr} with bounded $V$ that
decays exponentially. Define $u=e^{-|x|}$ in $\{ |x| > 1\}$ and extend it to a  $C^2$ smooth positive  function on the plane. Then $|\Delta u| \leq C|u|$ and by taking $u(\frac{1}{\sqrt C} \cdot)$ in place of $u$ one can make $|V|\leq 1$ in this example.

 The assumption that $u$ is real-valued is redundant because in the case of real-valued $V$
the real and imaginary parts of $u$ also satisfy \eqref{eq:schr}. But in the proof we will use that $u$ is real-valued.
The proof of Theorem \ref{main} combines the technique of quasiconformal mappings with two tricks. The tricks involve  nodal sets (zero sets) of $u$ and holes that are made in nodal domains (connected components of the complement of the zero set).  We describe the idea in Section \ref{Idea}. Some two-dimensional tools are used in the proof and the Landis conjecture in higher dimensions is still open.

Our second result is a local version of Landis' conjecture.

\begin{theorem} \label{local 1} Let $u$ be a real solution to $\Delta u+ Vu=0$ in $B(0,2R)\subset\mathbb{R}^2$, where $V$ is real-valued and $|V|\leq 1$. 
Suppose that $|u(0)|=\sup\limits_{B(0,2R)} |u|= 1 $. Then
for any $x_0$ with $|x_0|=R/2>2$, we have
$$ \sup\limits_{B(x_0,1)}|u| \geq \exp(-C R \log^{3/2}R)$$
with some absolute constant $C>0$.
\end{theorem}

The previous best known bound $\sup\limits_{B(x_0,1)}|u| \geq  \exp(-CR^{4/3}\log R)$,  was obtained in any dimension by Bourgain and Kenig \cite{BK05} in their proof of Anderson localization for the Bernoulli model, see also \cite{C05}.

Theorem  \ref{local 1} follows from the main local Theorem \ref{local main}, where we don't assume that $|u(0)|=\sup\limits_{B(0,2R)} |u|= 1 $, and prove a version of the three balls inequality.

Landis' conjecture was a subject to an extensive study. Under additional assumptions on $V$, some versions of Landis' conjecture are known, see \cite{B12},\cite{D19}, \cite{DKW19},\cite{EKPV},\cite{C05},\cite{C15},\cite{K98},\cite{Z16} and references therein. A related problem in a cylinder was studied in \cite{G81}.

\textbf{Notation.} By $c,C,C',... >0$ we denote various constants. Typically small constants are denoted by small letters and we  use capital letters for large constants. If a constant $C$ depends on a domain (or some other parameter), we say it.  Sometimes we state theorems without reminding that the functions are assumed to be real-valued and $u$ is a solution to \eqref{eq:schr} on $\mathbb{R}^2$.
A ball with center at $x$ of radius $r$ is denoted by $B(x,r)$ and the two-dimensional Lebesgue measure is denoted by $m_2$. 

\textbf{Acknowledgements.} The authors  
are grateful to Misha Sodin, Alexandru Ionescu, Charles Fefferman and Carlos Kenig 
for fruitful discussions.

This work was completed during the time A.L. served as a Clay Research Fellow and Packard Fellow.
E.M. was partially supported by NSF grant DMS-1956294 and by Research Council of Norway, Project 275113.
F.N. was partially supported by NSF grant DMS-1900008.

 \section{Strategy of the proof and local versions.} \label{Idea}
 The proof consists of three acts. First, we will explain the main ideas of each of them.
 
 \textbf{Description of Act I.} 
We will use  the following well-known fact  about nodal sets, which is proved in the Appendix (Lemma \ref{le: diameter}) for reader's convenience.  There is an absolute constant $r_0>0$ such that if $u$ is a solution to $\Delta u+ Vu$ in a neighborhood of a closed ball $\overline{B(z_0,r)}$ with $|V|\leq 1$, $u(z_0)=0$ and $0<r<r_0$, then the circle
 $ C(z_0,r)=\{ z:|z-z_0|=r\}$
 is intersecting  the zero set of $u$.
 
 It is also true that the singular set $$S=\{ x: u(x)=0 \textup{ and } \nabla u(x)= 0\}$$
 consists of isolated points and  the nodal set 
 $$F_0=\{x:u(x)=0\}$$ is a union of smooth curves, see \cite{CF85}.
 However the proof will not use it, but this structural result about nodal sets makes it easier to think about them.

 Now, assume that $u$ is a solution to \eqref{eq:schr} in $B(0,R)$, $R>1$. Take $\varepsilon>0$ (a small parameter to be chosen later) and add finitely many $C\varepsilon$ -- separated closed  disks of radius $\varepsilon$ to $F_0$ so that the distance from each disk to $F_0$ is $\geq C\varepsilon$
 and 
 $$F_0\cup  \textup{ union of the disks } \cup \{ z:|z|\geq R\}$$
  is a $3C\varepsilon$ -- net on the plane (assume $C>2$). Let us denote by $F_1$ the union of the closed disks, see Figure \ref{puncture fig}.
  
  \begin{figure}[h!]
   \includegraphics[width=0.8\textwidth]{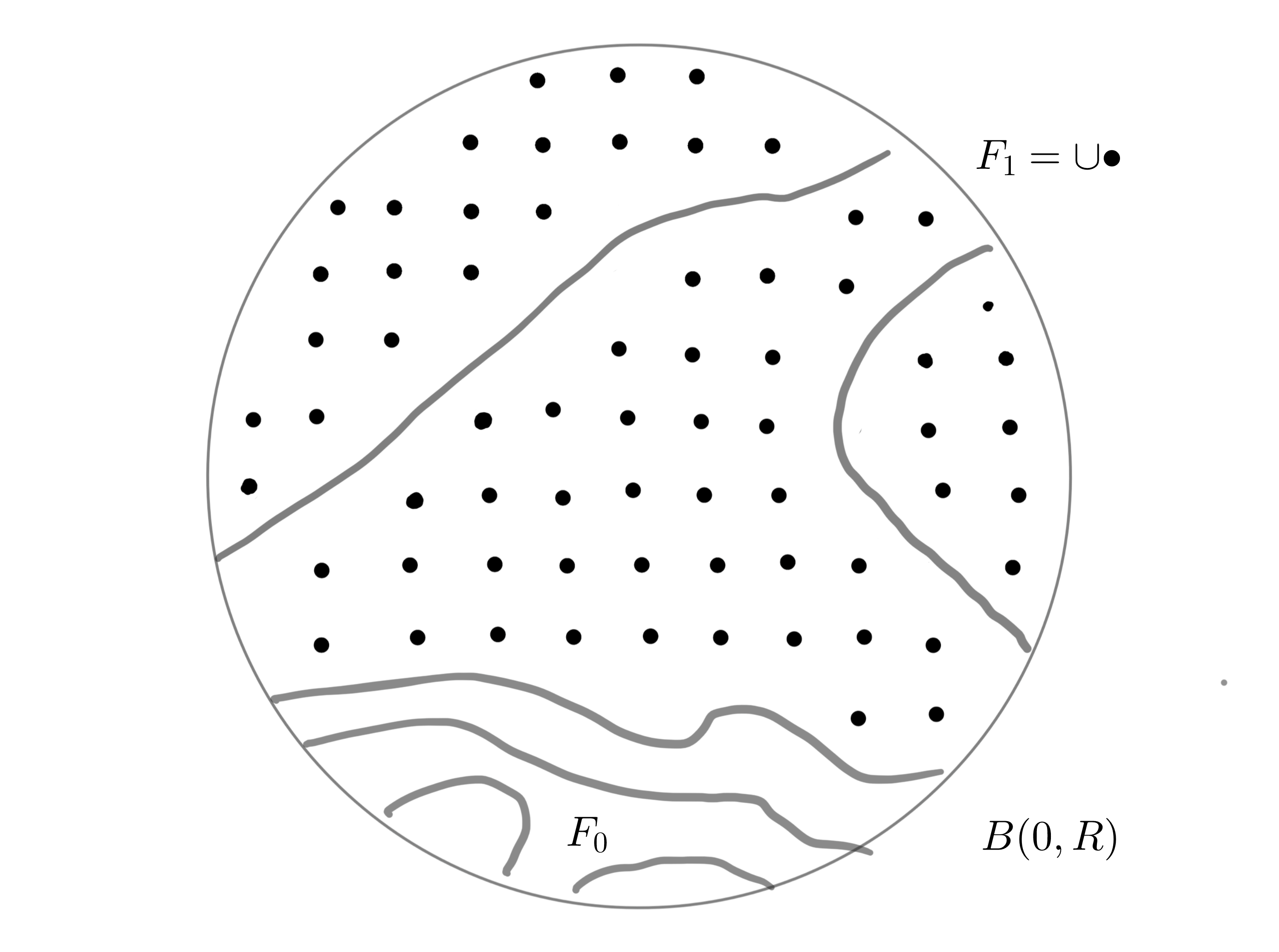}
   \caption{Puncturing nodal domains} \label{puncture fig}
\end{figure}

  It can be shown that 
  $$\Omega=\{z:|z| < R, z \notin F_0\cup F_1 \}$$
  is an open (possibly disconnected) set with the Poincare constant $\leq C'\varepsilon^2$, i.e., for every $u \in W_0^{1,2}(\Omega)$,
  we have 
  $$ \int_{\Omega} u^2 \leq C'\varepsilon^2 \int_{\Omega} |\nabla u|^2.$$ 
  \newpage
It allows one to construct a function $\varphi$ in $B(0,R)$ such that 
\begin{itemize}
\item $\Delta \varphi + V \varphi=0$ in $\Omega$,
\item $\varphi - 1 \in W_0^{1,2}(\Omega)$,
\item $\| \varphi-1\|_{\infty} \leq C''\varepsilon^2$.
\end{itemize}
The details are given in Section \ref{sec:Poincare}.

\textbf{Description of Act II.}
Consider $f= \frac{u}{\varphi}$. Then $f$ satisfies 
$$ {\rm{div}}(\varphi^2 \nabla f)=0$$
in $\Omega$. The set $\Omega$ is usually not connected and  the functions $\varphi$ and $f$  may be not smooth across $F_0$. However due to the fact that $F_0$ is the zero set of $u$ it appears (after some work) that the equation ${\rm{div}}(\varphi^2 \nabla f)=0$ holds through  $F_0$ in the whole $B(0,R)\setminus F_1$. 

 Here the theory of quasiconformal mappings joins the game.
After noticing that $f \in W_{loc}^{1,2}$, we may use the Stoilow factorization theorem  to make a $K$-- quasiconformal change of variables $g$ mapping $0$ to $0$ and 
$B(0,R)$ onto $B(0,R)$ such that 
$$ f =h \circ g$$
where $h $ is a harmonic function in $B(0,R)\setminus g(F_1)$. 
 Moreover, $K$ is very close to $1$ when $\| \varphi - 1 \|_\infty$  is small:  
$$ K \leq \frac{1+\left\|\frac{1-\varphi^2}{1+\varphi^2} \right\|_\infty}{1-\left \|\frac{1-\varphi^2}{1+\varphi^2}  \right \|_\infty}  \leq 1+C \varepsilon^2.$$
Mori's theorem tells us how much the distances are distorted depending on $K$:
\begin{equation} 
 \frac{1}{16}\left|\frac{z_1 - z_2}{R}\right|^{K} \leq \frac{|g(z_1) - g(z_2)|}{R}
  \leq 16\left|\frac{z_1 - z_2}{R}\right|^{1/K}.
\end{equation}
We choose $$ \varepsilon \sim \frac{1}{\sqrt{\log R}}$$
so that the distortion on scales from $
\frac{1}{R}$ to $R$ is bounded and, moreover, the images of the disks in $F_1$ have  size comparable to $\varepsilon$. 

Then we get a harmonic function $h$ in $B(0,R)\setminus g(F_1)$, where $g(F_1)$ is the union of sets of diameter $\sim \varepsilon$ and each set
(the image of a single disk) is surrounded by an annulus of 
width $\sim C\varepsilon$ in which $h$ does not change sign.

\textbf{Description of Act III.}
By rescaling we get the following question:
Let $h$ be harmonic in a punctured domain $B(0,R') \setminus \cup_j D_j$ where $R' \sim \frac{R}{\varepsilon}\sim R \sqrt{\log R}$ and $D_j$ are $1000$-- separated unit disks. Assume also that $h$ does not change sign  in $5D_j\setminus D_j$. What can be said about the decay of $|h|$?

\begin{theorem} \label{exercise 1} Under the above assumptions, we have

$$ \sup_{B(0,R') \setminus \cup_j 3D_j}|h| \leq \exp(CR')
\sup_{ \{ z: R'/8<|z|<R' \} \setminus \cup_j 3D_j} |h| \quad \textup{for } R' > 2000$$
with some absolute constant $C>0$.
\end{theorem}
Theorem \ref{exercise 1} is an immediate consequence of a more general Theorem \ref{local 3}.
The outcome is that $|u|$ cannot decay faster than $\exp(-CR\sqrt{\log R})$.
A different proof of the estimate for harmonic functions in a punctured domain (with a slightly worse bound)
is given in the Appendix. The second proof works in higher dimensions and uses the Carleman inequality with log linear weight.

\textbf{Local versions.}
Local versions of Theorem \ref{main} (on the two dimensional plane) are also true. Here is the main local Theorem \ref{local main}.

\begin{theorem} \label{local main}
If $u$ is a solution to $\Delta u + V u =0$ in $B(0, R)$, $R>2$, $V$ is real-valued,
 $|V| \leq 1$, and 
$$\frac{\sup_{B(0,R)}|u|}{\sup_{B(0,R/2)}|u|} \leq e^N, $$
then 
\begin{equation} \label{eq: Br}
\sup_{B(0,r)}|u| \geq  (r/R)^{C (R \log^{1/2}R +N)}\sup_{B(0,R)}|u|
\end{equation}
for any $r<R/4$, where $C$ is an absolute positive constant.
\end{theorem}
Theorem \ref{local main} implies Theorems \ref{main} and \ref{local 1}. 
In order to deduce Theorem \ref{main}, we may assume that $|u|$ attains its global maximum at some point on the plane, otherwise $|u|$ does not tend to $0$ near infinity. 
Let $$|u(z_{max})|= \max_{\mathbb{R}^2} |u|=1.$$ Then for any $ R > 6|z_{max}|$ and any $x$ with $|x|=R/3$, we have $$\sup_{B(x,R)}|u| = \sup_{B(x,R/2)}|u|=1$$ and if additionally $R>2$, then by Theorem \ref{local main} applied to $u(\cdot + x)$, we have $$\sup_{B(x,R/4)}|u| \geq e^{-C R \log^{1/2} R}$$ and therefore $$\sup_{|z|>R/12}|u| \geq e^{-C R \log^{1/2} R}.$$
In order to deduce Theorem \ref{local 1} note that $$\sup_{B(x,R/2)}|u| = \sup_{B(x,R)}|u|=1$$
for any $x$ with $|x|=R/2$ because $|u(0)|=\max_{B(0,2R)} |u|$. Applying Theorem \ref{local main}  to $u(\cdot + x)$ we get $$\sup\limits_{B(x,1)}|u| \geq R^{-C R\log^{1/2} R }= e^{-C R\log^{3/2}R}.$$

\begin{corollary} \label{local 2}
Let $A>4$. If $u$ is a solution to $\Delta u + V u =0$ in $B(0, 1)$, $V$ is real-valued,
 $|V| \leq A$, and  
$$\frac{\sup_{B(0,1)}|u|}{\sup_{B(0,1/2)}|u|} \leq \exp(N), $$
 then
\begin{equation} \label{eq:B_r}
\sup_{B(0,r)}|u| \geq  r^{C( \sqrt{A \log A}+N)}\sup_{B(0,1)}|u| \quad \text{ for } r\leq 1/4,
\end{equation}
 where $C$ is an absolute positive constant.
\end{corollary}

For the proof, consider $u(\frac{1}{\sqrt A}\cdot)$ in place of $u$. We obtain a solution to  $\Delta u+ Vu=0$
in $B(0,\sqrt A)$ with $|V| \leq 1$ and
$$\frac{\sup_{B(0,\sqrt A)}|u|}{\sup_{B(0,\sqrt A /2)}|u|} \leq e^{N}$$
and we can apply Theorem \ref{local main} to the new $u$ and $R=\sqrt A$.

\begin{remark}
Inequality \eqref{eq:B_r} implies that the vanishing order of $u$ at $0$ is bounded by $C (\sqrt{A \log A}+N)$.
This question was previously studied in \cite{B12},\cite{K98},\cite{Z16}.
\end{remark}
On any smooth two dimensional Riemannian manifold $(M,g)$ every equation $\Delta_g u+ Vu=0$ can
be simplified in local isothermal coordinates to $\Delta u + V'u=0$ (with ordinary Euclidean Laplacian $\Delta$). Corollary \ref{local 2} gives information on the distribution of solutions to Schrodinger equations on compact  manifolds of dimension 2.
\begin{corollary}\label{global} Let $(M,g)$ be a smooth closed (compact and without boundary) Riemannian manifold of dimension $2$.  Then for any function $u$ satisfying $\Delta_g u + V u =0$ on $M$ with $|V| \leq \lambda$, $\lambda >2$, we have
 $$\sup\limits_{B_r}|u| \geq  r^{C \sqrt{\lambda \log \lambda}}\sup_{M}|u|$$
 for any  ball $B_r$ of radius $r<1/2$. The constant $ C$ depends on the manifold.
\end{corollary}

This result follows from Corollary \ref{local 2}  by iterations (see the argument in \cite{DF}, page 162, after formula (1.5)).
\begin{remark}
A slightly better bound was obtained in \cite{DF} by Donnelly and Fefferman for Laplace eigenfunctions on closed Riemannian manifolds of any dimension. If $\Delta_g u+ \lambda u=0$ on $(M,g)$, then
$$ \sup_{B_r} |u| \geq c  r^{C \sqrt \lambda }  \sup_M |u|, \quad r \leq \frac 1 2.$$
So the vanishing order at any point is at most $C\sqrt \lambda$.
\end{remark}

In Act I and Act II we will reduce (with a logarithmic loss) the main local Theorem \ref{local main}
to a general Theorem \ref{local 3}, which is  a local  statement about two dimensional harmonic functions.

\section{Act I} \label{sec:Poincare}

\subsection{Poincare constant for porous domains.}

\begin{lemma} \label{Poincare}
Let $F$ be a closed set in $B(0,R)$, $R>1$, such that
\begin{enumerate}
\item[a)] For every $z_0 \in F$, $r\in (0,1]$,
the circle
 $ C(z_0,r)=\{ z:|z-z_0|=r\}$
 intersects  $F\cup \partial B(0,R)$. 
\item[b)] $ F \cup \partial B(0,R)$ is $C-$dense in $B(0,R)$, $C>1$.
\end{enumerate}
Then the Poincare constant of $\Omega= B(0,R)\setminus F$
is bounded by some constant $\widetilde C$ that depends only on $C$.
\end{lemma}
\begin{proof}
Let $f \in C_0^\infty(\Omega)$. Extend $f$ by zero outside $\Omega$.
First, we will show that if $z\in F\cup \partial B(0,R)$, then
$$ \int_{B(z,3C)} |f|^2 \lesssim \int_{B(z,3C)} |\nabla f|^2 .$$
Every circle $C_r=\partial B(z,r)$, $r\in(0,1)$, has a zero of $f$, whence
$$\max\limits_{C_r}|f| \leq \int\limits_{C_r }|\nabla f|$$
and 
$$\int\limits_{B(z,1)}|f|^2 = \int_0^1 \left(\,\int\limits_{C_r}|f|^2\right) dr \leq \int_0^1 |C_r|\max\limits_{C_r}|f|^2dr \leq  \int_0^1 |C_r|\left(\,\int\limits_{C_r }|\nabla f|\right)^2 dr \leq$$
$$ \leq \int_0^1 |C_r|^2\left(\,\int\limits_{C_r }|\nabla f|^2\right) dr \leq (2\pi)^2 \int_0^1 \left(\,\int\limits_{C_r }|\nabla f|^2\right)dr = (2\pi)^2 \int\limits_{B(z,1)} |\nabla f|^2.  $$
We therefore can find $r\in (1/2,1)$ such that 
$$ \int\limits_{C_r }|f|^2 \leq C_1 \int\limits_{B(z,1)}|f|^2 \leq C_2 \int\limits_{B(z,1)} |\nabla f|^2.$$
Let $\Gamma_\psi$, $\psi \in[0,2\pi)$, be a segment starting at the point $$x_\psi:= z + re^{i\psi}$$
and ending at the point $z+3Ce^{i\psi}$. Note that
$$ \max_{\Gamma_\psi} |f|^2 \leq \left(|f(x_\psi)| + \int_{\Gamma_\psi}|\nabla f|\right)^2 \leq 2|f(x_\psi)|^2+2\left(\int_{\Gamma_\psi}|\nabla f|\right)^2 \leq$$
$$ \leq 2 |f(x_\psi)|^2 + 2|\Gamma_\psi|\int_{\Gamma_\psi}|\nabla f|^2 \leq 2 |f(x_\psi)|^2 + 6C \int_{\Gamma_\psi}|\nabla f|^2$$
and therefore
$$ \int\limits_{ B(z,3C)\setminus B(z,1)} f^2 \leq  (3C)^2 \int_0^{2\pi} \max_{\Gamma_\psi} |f|^2 d\psi  \leq $$ $$ \leq C_1(C) \left[ \quad \int\limits_{C_r }|f|^2 +  \int_0^{2\pi} \left(\int_{\Gamma_\psi}|\nabla f|^2\right) d\psi \right] \leq C_2(C) \int_{B(z,3C)} |\nabla f|^2. $$
Thus  $$ \int_{B(z,3C)} |f|^2 \leq C_3(C) \int_{B(z,3C)} |\nabla f|^2 .$$

We can choose a finite collection $Z_*$ of points $z$ in $F \cup \partial B(0,R)$ such that the balls $B(z,3C)$ cover $B(0,R)$ and
each point is covered a bounded number of times.
Finally, we have
$$ \int\limits_{B(0,R)} f^2 \leq \sum\limits_{z\in Z_*} \int_{B(z,3C)} |f|^2 \leq C_3(C) \sum\limits_{z\in Z_*} \int_{B(z,3C)} |\nabla f|^2 \leq $$ 
$$\leq C_4(C) \int\limits_{B(0,R)} |\nabla f|^2.$$
\end{proof}

 We start proving Theorem \ref{local main}.
Recall that $\Delta u + Vu =0$ in the ball $B(0,R)$ (we may think that $R$ is a large number) and $F_0$ is the zero set of $u$. 
We will use the fact that $u\in C^1(B(0,R))$, which is proved in the Appendix, see Fact \ref{fact5}.
 Now, consider the following setting:
 \begin{figure}[h!] \label{contours}
    \includegraphics[width=0.8\textwidth]{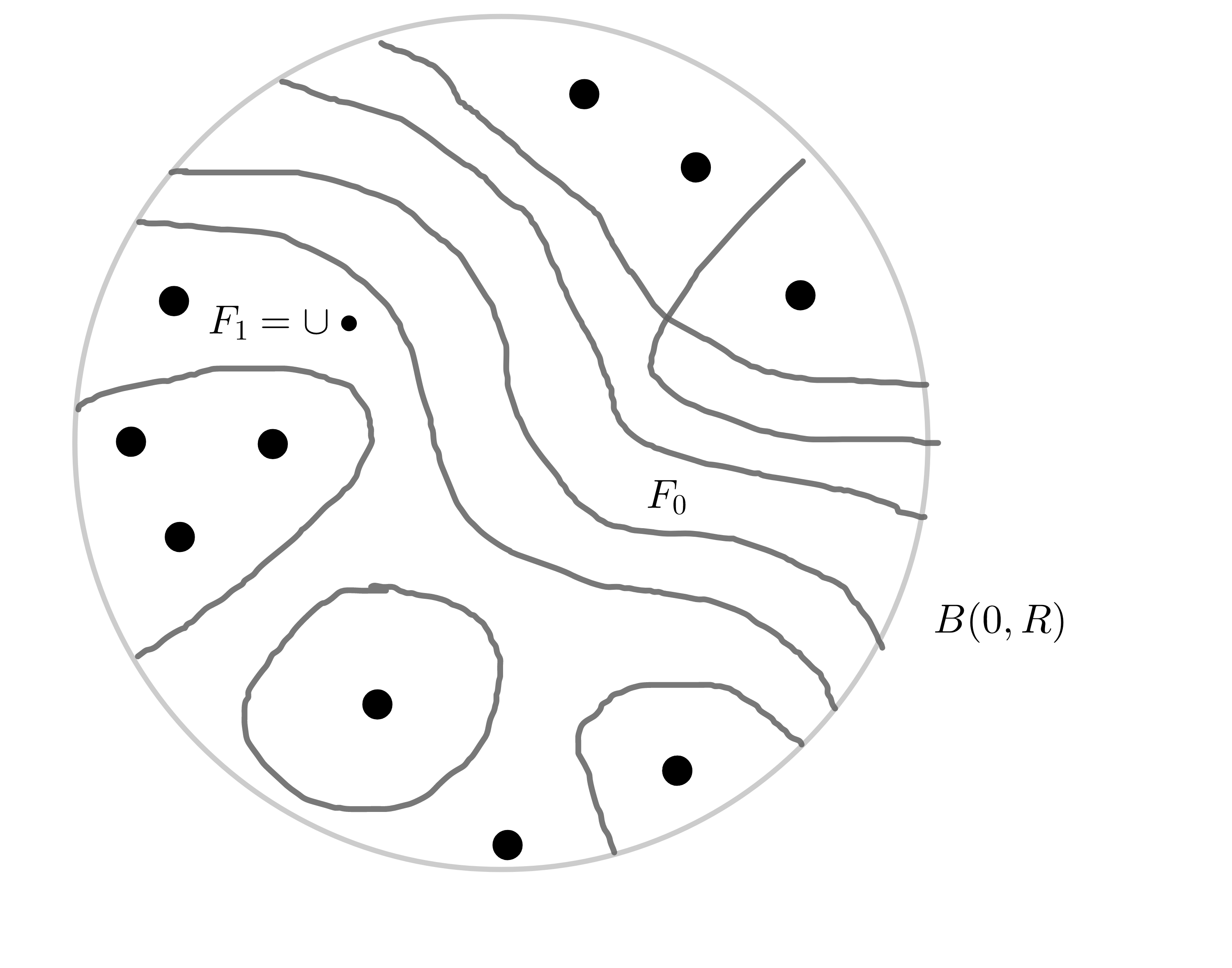}
   \caption{Puncturing nodal domains}
\end{figure}

\noindent Take $\varepsilon>0$ (a small parameter to be chosen later). Choose finitely many $C\varepsilon$ -- separated closed disks of radius $\varepsilon$, whose union will be denoted by $F_1$,  so that the distance from each disk to $F_0$ and $\partial B(0,R)$ is $\geq C\varepsilon$
 and  $$F_0\cup F_1\cup \partial B(0,R)$$
  is a $3C\varepsilon$ -- net in $B(0,R)$ (we assume $C>2$). \\

   For instance, one can get $F_1$ by 
  considering  the maximal number of open non-intersecting disks of radius $(C+1)\varepsilon$ in $B(0,R)\setminus F_0 $.  The centers of the disks are $(2C+2)\varepsilon$ -- separated.  There is no point $x$ in $B(0,R)\setminus F_0$ that is $(2C+2)\varepsilon$ far from the centers of the disks and from $F_0 \cup \partial B(0,R)$, otherwise we could add one more disk of radius $(C+1)\varepsilon$ with center at this point. So we may choose the disks of radius $\varepsilon>0$ with the same centers, they will be $C\varepsilon$ -- separated and $F_0\cup F_1\cup \partial B(0,R)$ will be a $2(C+1)\varepsilon$ -- net.
  
 \noindent \textbf{Two points to avoid.} Now, let us remove from $F_1$ the disks that are $C\varepsilon$ close to $0$ or to the point $z_{\max} \in \overline{B(0,R/2)}$ such that $$|u(z_{\max})|= \sup_{B(0,R/2)}|u|.$$ The set $F_0\cup F_1\cup  \partial B(0,R) $ will still be a $10C\varepsilon$ -- net, but now all disks from $F_1$ are also  $C\varepsilon$- separated from $0$ and $z_{\max}$.
  The detail about avoiding those two points will be used only in the end of Act II.
 
 Recall that $F_0$ has the property that for any $z_0 \in F_0$, every circle $C(z_0,r)$ with $r<r_0$ intersects
  $F_0$ or $\partial B_R$. Taking $u(\varepsilon \cdot)$ in place of $u$ (so the assumptions of Lemma \ref{Poincare} hold for $\varepsilon< r_0$) and applying Lemma \ref{Poincare} we arrive to the following conclusion. \\
 \noindent\textbf{Outcome.} 
 The domain $$\Omega= B(0,R)\setminus (F_0 \cup F_1)$$
   has Poincare constant $\leq  C'\varepsilon^2$ and $B(0,R)\setminus F_1$ contains $0$ and $z_{\max}$.
  
\subsection{Solving $\Delta \varphi+ V\varphi=0$.}
The goal of this section is to construct an auxiliary solution to \eqref{eq:schr} in a domain with a small Poincare constant, so that the solution has boundary values $1$ and is uniformly close to $1$. 
\begin{lemma} \label{solving Schrodinger}
Let $\Omega$ be a bounded open set with the Poincare constant $k^2$. Let $V\in L^\infty(\Omega)$. Assume that 

$$k^2 \|V\|_\infty \ll 1.$$
Then there exists $\varphi=1+\tilde \varphi$ with
 $$\tilde \varphi \in W_{0}^{1,2}(\Omega), \|\tilde \varphi\|_\infty \leq Ck^2\|V\|_\infty$$
 such that $\varphi$ is a weak solution to $\Delta \varphi + V\varphi=0$ in $\Omega$, where $C$ is an absolute positive constant.
 \end{lemma}
 \begin{proof}
 
 We will use the following fact, which is proved in the Appendix, Lemma \ref{lem: solving} and Lemma \ref{lem: solving2}.
 
\noindent \textbf{Fact.} When the Poincare constant  of $\Omega$ is $ 1$, $v\in L^{\infty}(\Omega)$, there is a solution $\varphi$ to $\Delta \varphi = v$ in $W_{0}^{1,2}(\Omega)$ with $$\| \varphi\|_\infty \leq C\| v\|_\infty$$
 and 
 $$ \| \varphi \|_{W_0^{1,2}} \leq C \|v\|_2.$$

\noindent \textbf{Corollary} (follows by rescaling). If $\Omega$ has Poincare constant $k^2$, then we can find a solution $\varphi$ to $\Delta \varphi = v$ with
$$\| \varphi \|_\infty \leq Ck^2\| v\|_\infty$$
and 
$$ \| \varphi \|_{W_0^{1,2}} \leq  C_1(k) \|v\|_2. $$

Now, let $\varphi_1$ solve $\Delta \varphi_1= -V$ and for $n\geq 2$ 
let $\varphi_n$ solve 
$$\Delta \varphi_n= -V\varphi_{n-1}.$$
Note that this sequence is well defined since on each step the right-hand side is in $L^\infty$.
 We have $$\|\varphi_n\|_\infty\leq Ck^2\|V\|_\infty\|\varphi_{n-1}\|_\infty,\quad n\geq 2,$$
 and $\|\varphi_1\|_\infty \leq Ck^2 \|V\|_\infty.$ 
 We are assuming that $Ck^2 \|V\|_\infty \leq 1/2$.
 Hence $\|\varphi_n\|_\infty \leq 2^{-n+1} Ck^2\|V\|_\infty $ 
 and 
 $$\|\varphi_n\|_{W_0^{1,2}} \leq  C_1(k) \|\varphi_{n-1}\|_2 \leq C_2(k) \|\varphi_{n-1}\|_{\infty}  \leq C_3(k) 2^{-n}.$$
 
 Thus the series $$\tilde \varphi = \varphi_1+\varphi_2+\dots$$
 converges both in $L^{\infty}$ and in $W_0^{1,2}(\Omega)$
 with $$\| \tilde \varphi\|_{\infty} \leq C'k^2 \|V\|_\infty.$$ 
 Also  for any  $h \in W_0^{1,2}(\Omega)$, we have $\int \nabla \varphi_n \nabla h = \int V \varphi_{n-1}h$ for $n \geq 2$ and $\int \nabla \varphi_1 \nabla h = \int Vh$.
 Thus $\Delta \tilde \varphi= - V(1+\tilde \varphi)$ and 
 $$\Delta(1+\tilde \varphi) + V(1+\tilde \varphi) =0 \quad \text{ in } \Omega $$  as required.
 
 \end{proof}

\noindent \textbf{Outcome.} Since the Poincare constant of $\Omega= B(0,R) \setminus (F_0 \cup F_1)$ is $\leq \widetilde C\varepsilon^2$,  using Lemma \ref{solving Schrodinger}, we can find  $\varphi$ such that
\begin{itemize}
\item $\Delta \varphi + V \varphi=0$ in $\Omega$,
\item $\varphi - 1 \in W_0^{1,2}(\Omega)$,
\item $\| \varphi-1\|_{\infty} \leq C'\varepsilon^2$.
\end{itemize}

\section{Act II.}

\subsection{Reduction to a divergence type equation in a domain with holes.}
 Recall that $u$ is a solution to $\Delta u+ Vu=0$ in $B(0,R)$ and $F_0$ is the zero set of $u$.
 Extend the function $\varphi$   by 1 outside $$\Omega=B(0,R) \setminus (F_0\cup F_1).$$

\begin{lemma} \label{div equation}

The function $\frac{u}{\varphi} \in W^{1,2}_{loc}(B(0,R))$ and it is a solution to $$\textup{ div}(\varphi^2 \nabla (\frac{u}{\varphi})) = 0$$ in 
$B(0,R) \setminus  F_1$ in the weak sense.

\end{lemma}
\noindent \textbf{Remark.} The  lemma takes care of all ``continuations through nodal lines" of $u$.
\begin{proof}

 First, we would like to notice that the extended functions
  $\frac{1}{\varphi}, \varphi \in W^{1,2}_{loc}(\mathbb{R}^2)$ and
 \begin{equation} \label{grad 1/phi}
  \nabla \frac{1}{\varphi}=  -\mathbbm{1}_{\Omega}  \frac{\nabla \varphi}{\varphi^2} \quad \text{and} \quad \nabla \varphi = \mathbbm{1}_\Omega \nabla \varphi
  \end{equation}
in $\mathbb{R}^2$ in the  sense of distributions:
$$ \int_{\mathbb{R}^2} \frac{1}{\varphi} \nabla \xi= \int_{\Omega}\frac{\nabla \varphi}{\varphi^2} \xi \quad \text{and} \quad \int_{\mathbb{R}^2} \varphi \nabla \xi= -\int_{\Omega}\nabla \varphi \xi$$
for any $\xi
\in C^\infty_0(\mathbb{R}^2) $. The formal check is performed in Fact \ref{fact6} in the Appendix.

Now, we would like to verify that $ \frac{u}{\varphi} \in W^{1,2}_{loc}(B(0,R))$ and
 $$ \nabla \frac{u}{\varphi} = \frac{\varphi\nabla u}{\varphi^2} - \frac{u \nabla \varphi }{\varphi^2} \mathbbm{1}_\Omega.$$

  \begin{fact} \label{product}
  Let $u,v \in W^{1,2}_{loc}(B(0,R))\cap L^{{}^{\scriptsize \infty}}_{loc}(B(0,R))$. Then $uv \in W^{1,2}_{loc}(
B(0,R))  $ and $\nabla(uv) = u \nabla  v + v \nabla u$.
  \end{fact}
\noindent Fact \ref{product} is proved in the Appendix.


Recall that $\varphi$ is extended by $1$ outside $\Omega$, $\frac{1}{\varphi} \in W^{1,2}_{loc}(\mathbb{R}^2)$
and $u$ is $C^1$-smooth in $B(0,R)$ by Fact \ref{fact5}.
By Fact \ref{product} we know that $ \frac{u}{\varphi} \in W^{1,2}_{loc}(B(0,R))$ and, as expected,
$$ \nabla \frac{u}{\varphi} = \frac{\varphi \nabla u}{\varphi^2} - \frac{u \nabla \varphi}{\varphi^2} \mathbbm{1}_\Omega$$
in $B(0,R)$ in the sense of distributions. To establish the divergence-type equation for $\nabla \frac{u}{\varphi}$ 
we want to show that for every  test function $h \in C_0^\infty(B(0,R)\setminus F_1)$, we have
$$ \int_{B(0,R)\setminus F_1} \varphi^2 \nabla (\frac{u}{\varphi}) \nabla h =0.$$
 So we need to prove that
 \begin{equation} \label{need}
 \int_{B(0,R)\setminus F_1} (\varphi \nabla u - u \nabla \varphi \mathbbm{1}_\Omega) \cdot  \nabla h =0.
 \end{equation}
 Since $u$ is a solution to $\nabla u +Vu =0$ in $B(0,R)$,
 we have 
  \begin{equation} \label{have1}
 \int_{B(0,R)\setminus F_1} \nabla u \cdot (\varphi \nabla h + h \nabla \varphi) = \int_{B(0,R)\setminus F_1} V\varphi u h
 \end{equation}
 (we know the last equality under the assumption that $
 \varphi$ is smooth, but it is also true for  $\varphi \in W^{1,2}_{loc}(\mathbb{R}^2)$  by taking the norm limit).
  Consider a function $\xi \in C_0^\infty(B(0,R) \setminus F_0)$ that descends from $1$ to $0$ in 
 the $\varepsilon$ -- neighborhood of $F_0\cup \partial B(0,R)$ with $|\nabla \xi| < C/\varepsilon$.
  
  Since $\Delta \varphi + V \varphi =0$ in $$\Omega= B(0,R) \setminus (F_0 \cup F_1)$$ and $uh\xi \in C_0^1(\Omega)$, we have
  
  \begin{equation} \label{have2}
   \int_\Omega \nabla \varphi \cdot(h \nabla u \xi +u\nabla h \xi + uh \nabla \xi)=  \int_\Omega V\varphi uh\xi.
  \end{equation}
 Note that $\int_\Omega V\varphi uh\xi$ tends to $\int_\Omega V\varphi u h$ as $\varepsilon \to 0$ (the functions $V,\varphi,uh,\xi$ are uniformly bounded and 
the convergence  holds pointwise in $\Omega$ because  $\xi \to 1 $ in $B(0,R)\setminus F_0$). 
   Note that
  $$h \nabla u \xi \to h\nabla u \quad \text{pointwise in } \Omega$$
   and 
   $$u\nabla h \xi \to  u\nabla h \quad \text{pointwise in } \Omega $$ 
  because $\xi \to 1$ in $\Omega$.
   Hence
  \begin{equation} \label{have3} \int_\Omega \nabla \varphi \cdot(h \nabla u \xi +u\nabla h \xi) \to \int_\Omega \nabla \varphi (h\nabla u + u\nabla h) \quad \text {as } \varepsilon \to 0
 \end{equation}
 by the Lebesgue dominated convergence theorem with the majorant $|\nabla \varphi| (|h||\nabla u| + |u||\nabla h|)$.

 In order to prove  \eqref{need}  we will show
 that 
 $$ \int_{\Omega} \nabla \varphi \cdot (uh\nabla \xi) \to 0.$$
 And here is the main place where we use that $F_0$ is the zero set of $u$!
 Note that $uh \in C^1_0(B(0,R))$ and vanishes on $F_0$, so $|uh| \leq C_1(u,h)\varepsilon$ in the $\varepsilon$-- neighborhood  of the zero set of $u$.  Thus $|uh\nabla \xi|$ is  bounded by some constant $C(u,h)$ in $B(0,R)$. Also $m_2(\textup{supp} \nabla \xi)$ goes to 0. Hence

$$ \int_\Omega \nabla \varphi \cdot  (hu\nabla \xi) \leq C(u,h) \sqrt{m_2(\textup{supp} \nabla \xi)} \sqrt{\int_\Omega |\nabla \varphi|^2} \to 0.$$
 By \eqref{have2},\eqref{have3} we obtain
 $$\int_{\Omega} \nabla \varphi \cdot(h \nabla u  +u\nabla h) =  \int_{\Omega} V\varphi uh=\int_{B(0,R)\setminus F_1} V\varphi uh$$
 (the second equality is due to the fact that $u=0$ on $F_0$).
  Using $\nabla \varphi = \nabla \varphi \mathbbm{1}_\Omega$ in the sense of distributions, we have
\begin{equation} \label{have4} 
 \int_\Omega \nabla \varphi (h\nabla u + u\nabla h) = \int_{B(0,R)\setminus F_1} \nabla \varphi (h\nabla u + u\nabla h).
  \end{equation}
  Thus 
 $$\int_{B(0,R)\setminus F_1} \nabla \varphi \cdot(h \nabla u  +u\nabla h) =  \int_{B(0,R)\setminus F_1} V\varphi uh$$
 and, subtracting  \eqref{have1}, we finish the proof of \eqref{need}.
\end{proof}

\subsection{Quasiconformal change of variables.}

We briefly describe some facts from the theory of quasiconformal mappings, which are used in the study of the solutions to equations in divergence form on the plane, and explain why the solutions behave like ordinary harmonic functions.  We partially follow the exposition from \cite{NPS}, where the quasiconformal mappings are applied to quasi-symmetry of Laplace eigenfunctions.

Let $B$ be a disk on the plane. Consider a real-valued function  $f \in W^{1,2}_{loc}(B)$ satisfying
\begin{equation} \label{eq:div}
 \textup{div}(\varphi^2 \nabla f) = 0
 \end{equation}
 and assume that $ 0< c < \varphi(x) < C < +\infty$ in $B$.
One can find a function $\tilde f \in W^{1,2}_{loc}(B)$ such that
$$ \varphi^2 f_x = \tilde f_y  \textup{ and } \varphi^2 f_y = -\tilde f_x$$ 
(see Section \ref{sec: divergence}) and $f$ appears to be the real part of  $w=f+i\tilde f$. A direct computation shows that $w$ is a solution to the Beltrami equation:
\begin{equation} \label{eq:Beltrami}
  \frac{\partial w}{\partial \overline z} = \mu \frac{\partial w}{\partial  z}
\end{equation}  
with the Beltrami coefficient 
\begin{equation} \label{eq:mu}
\mu= \frac{1-\varphi^2}{1+\varphi^2}\cdot \frac{f_x+i f_y}{f_x-if_y}.
\end{equation}
When $\nabla f=0$, we put $\mu=0$.

We are going  to apply the theory of quasiconformal mappings in a situation when $f=\frac{u}{\varphi} $ and the domain $$\Omega_1:= B(0,R)\setminus F_1$$ is not simply connected. In this case $w$ and $\tilde f$ can be defined only locally, but not in the whole $\Omega_1$. However the Beltrami coefficient $\mu$  is  well defined by \eqref{eq:mu} in  $\Omega_1$ and 
$$|\mu| \leq \frac{1-\varphi^2}{1+\varphi^2} \leq C\varepsilon^2.$$
Let us extend $\mu$ by zero outside $\Omega_1$ to the whole complex plane. Now $\mu$ has a compact support.

The existence Theorem 5.3.2 \cite{AIM09} claims that there is a $K$-quasiconformal homeomorphism $\psi$ of the complex plane such that 
\begin{itemize}
\item $\psi \in W^{1,2}_{loc}$,
\item
 $ \frac{\partial \psi}{\partial \overline z} = \mu \frac{\partial \psi}{\partial  z}$,
\item  
 $ K \leq \frac{1+\sup |\mu|}{1 - \sup |\mu|}.$
 
\end{itemize}
In our case $$K \leq 1+ C' \varepsilon^2.$$
\textbf{Claim.} The function $f \circ \psi^{-1}$ is harmonic in $\psi(\Omega_1)$.

Indeed, for any ball $B \subset \Omega_1$, we can define $w\in W^{1,2}_{loc}(B)$ such that $f=\Re w$ and $w$,$\psi$ solve the same Beltrami equation.  Stoilow factorization theorem  (\cite{AIM09}, p.179, Theorem 5.5.1)
claims that 
there is a holomorphic function $W$ such that
$$ w = W(\psi(z))$$
and therefore the harmonic function $\Re W$ satisfies
$$ f(z)= \Re W(\psi(z)).$$
Clearly, the local observation shows that $f\circ \psi^{-1}$ is a harmonic function in $\psi(\Omega_1)$.

Note that $\psi(B(0,R))$ is a simply connected domain (and not the whole plane).
Using the Riemann uniformisation theorem we can find a conformal map that sends $\psi(B(0,R))$ back to $B(0,R)$ and $\psi(0)$ to $0$. The composition of this conformal map and the  $K$-quasiconformal homeomorphism $\psi$ will be a  $K$-quasiconformal homeomorphism 
$g$ of $B(0,R)$ onto itself with $g(0)=0$.
Then the function $h=f \circ g^{-1}$ is harmonic in $g(\Omega_1)$.

\textbf{Distortion of quasiconformal mappings.}
Mori's theorem (\cite{Ahl66}, Chapter III, Section C)
tells us that distances are changed by $g$ in a controlled way:
\begin{equation} \label{eq:distortion}
 \frac{1}{16}\left|\frac{z_1 - z_2}{R}\right|^{K} \leq \frac{|g(z_1) - g(z_2)|}{R} \leq 16\left|\frac{z_1 - z_2}{R}\right|^{1/K}.
\end{equation}

We choose $$ \varepsilon = \frac{c}{\sqrt{\log R}} $$
so that $$K\in [1,1+C c^2/ \log R),  \quad R^K\asymp R \asymp R^{1/K}$$ and the distortion on scales from $
\frac{1}{R}$ to $R$ is bounded. Namely, we may choose $c$ so small that if $\frac{1}{R}\leq|z_1-z_2| \leq2R$, then
$$\frac{1}{32}|z_1-z_2| \leq|g(z_1)-g(z_2)| \leq 32|z_1-z_2|.$$

Note that in the statement of Theorem \ref{local main} one can safely assume that $R$ is sufficiently large ($R\gg 1$) by rescaling, which makes $\|V\|_\infty$ only smaller. It is needed to make $\varepsilon \geq 1/R$.  Then we get a harmonic function $h$ in $B(0,R)\setminus g(F_1)$, where $g(F_1)$ is the union of sets of diameter $\sim \varepsilon$. The image of a single disk  of radius $\varepsilon$ will be contained in a disk of radius $32\varepsilon$. Let us denote these disks of radius $32\varepsilon$ by $D_j$. The images of disks from $F_1$ are $ \frac{C}{32}\varepsilon$ -- separated from each other and from the zero set of $h$.  Hence $D_j$ are  $ (\frac{C}{32}\varepsilon - 128\varepsilon)=C_132\varepsilon$ -- separated from each other and from the zero set of $h$, and $h$ does not change sign in $C_1D_j \setminus D_j$. We have $$C_1 = \frac{C}{32^2} - 4 >  100$$
if $C= 10^6$.

We specifically asked that $0$ and $z_{\max}$ (the point where $\sup_{B(0,R/2)}|u|$ is attained) are $C\varepsilon$ -- separated from the disks. Recall that $g(0)=0$, so the disks $C_1D_j$ do not contain $0$ and $g(z_{\max})$. The distortion estimate implies that
$g(z_{\max}) \in \overline{B(0,R-R/64)}$. Since  we had 
$$\frac{\sup_{B(0,R)}|u|}{\sup_{B(0,R/2)}|u|} \leq e^N,$$
we conclude that
$$ \frac{\sup_{B(0,R) \setminus \cup 3D_j}|h|}{\sup_{B(0,R-R/64)\setminus \cup 3D_j}|h|} \leq e^N.$$

If we  make the rescaling by a factor of $32 \varepsilon$, then the disks $D_j$ become 100-separated unit disks and $R$ becomes $$R'=R\cdot32\varepsilon\sim R\sqrt{\log R}.$$

The goal of Theorem \ref{local main} is to estimate $\sup_{B(0,r)}|u|$ from below.
If $r<1/R$, the image of $B(0,r)$ may have radius significantly smaller than $r$.
However $g(B(0,r))$ contains  a disk with center at $0$ of radius $$\frac{R}{16}\left(\frac{r}{R}\right)^{K} \geq \frac{R}{16} \left(\frac{r}{R}\right)^2.$$
Let $\tilde g = \frac{1}{32\varepsilon} g$. Then $\tilde g(B(0,r))$ contains a ball $B(0,r')$, where $$r' \geq
 \frac{R'}{16} \left(\frac{r}{R}\right)^2.$$
So $$\frac{R'}{r'} \leq 16  \frac{R^2}{r^2}.$$
In order to prove estimate \eqref{eq: Br}, it is enough to show that 
$$ \sup_{B(0,r')\setminus \cup 3D_j}|h| \geq c (r'/R')^{C (R' +N)}\sup_{B(0,R')\setminus \cup 3D_j}|h|. $$
It will be proved in Theorem \ref{local 3}.

\section{Act III}  \label{sec: toy}
Before we formulate and prove the promised local Theorem \ref{local 3} we will explain the main idea in the global case.
\begin{theorem}[Toy problem] \label{thm: toy}

 Let $\{D_j\}$ be a collection of $100$-separated disks with unit radius on the complex plane
 $\mathbb{C}$. Suppose that $u$ is a harmonic function in $\mathbb{C} \setminus \cup_j D_j$ which
 preserves sign in each annulus $5D_j\setminus D_j$. If $|u(z)| \leq e^{-L|z|}$ for all $z \in \mathbb{C} \setminus \cup_j D_j$ and $L$ is sufficiently large, then $u \equiv 0$.
 
 \end{theorem}
 \begin{proof} We start with a simple observation.
 
 \begin{claim}
 Let $m_j= \min\limits_{\partial 3D_j}|u|$. Then for some absolute
 constant $A>0$, we have
 \begin{enumerate}
 \item  $\max\limits_{\partial 3D_j}|u| \leq Am_j$,
 \item  $\max\limits_{\partial 3D_j}|\nabla u| \leq Am_j$.
 \end{enumerate}
 \end{claim}
 \begin{proof}
 By the Harnack inequality there exists a constant $A>0$ such that
 $$ \sup\limits_{4D_j\setminus 2 D_j}|u| \leq A\inf\limits_{4D_j\setminus 2 D_j}|u| \leq Am_j,$$
 which proves the first part of the claim. The second part follows from the Cauchy inequality.
 
 \end{proof}
  \newpage
 Let $k\in (0,L)$  and consider the numbers $m_j e^{k \Re z_j}$, where $z_j$ is the rightmost point of $3D_j$.
 
  \includegraphics[width=0.5\textwidth]{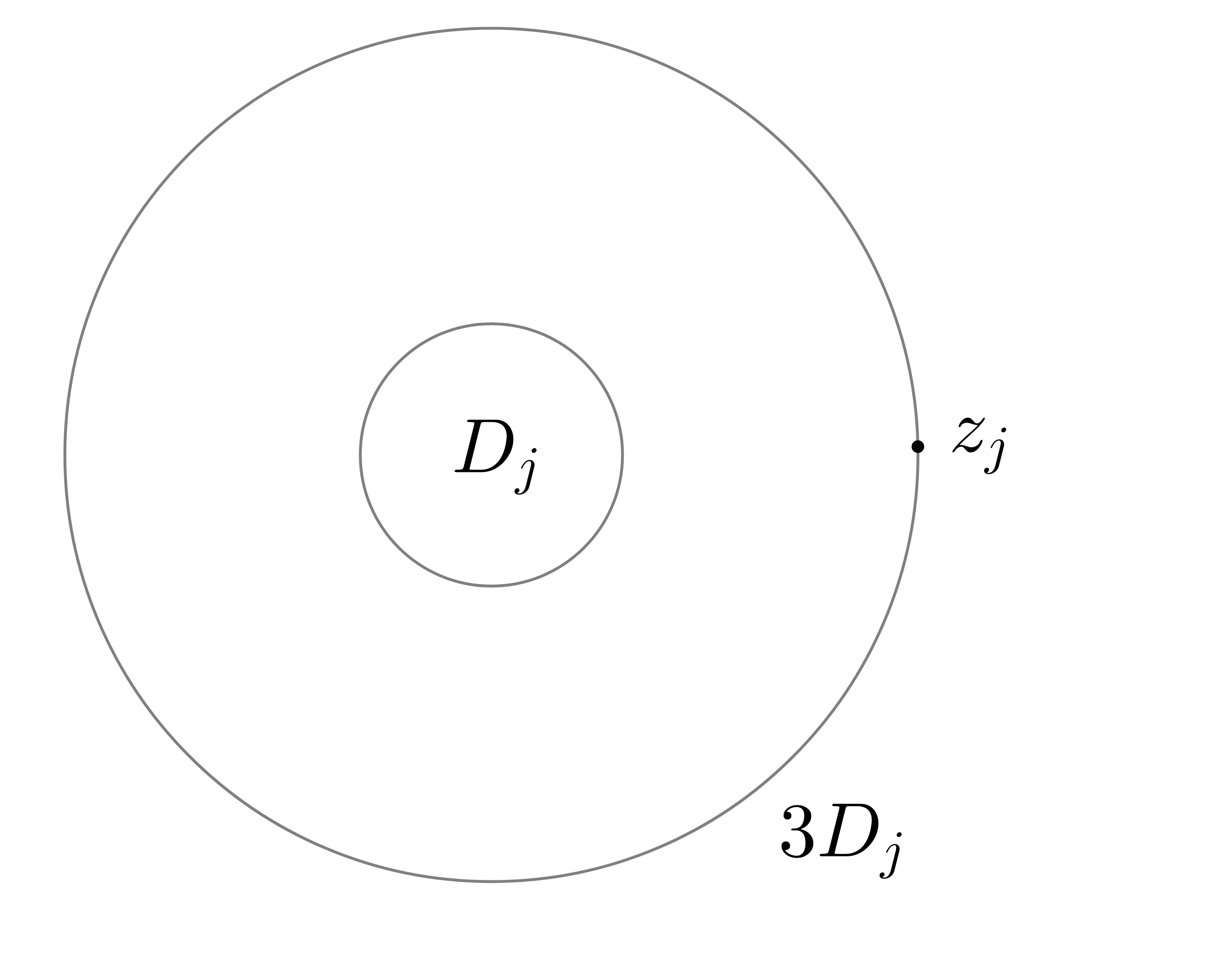}
  
 \noindent Since $$m_j \leq |u(z_j)| \leq e^{-L|z_j|},$$
  there is $j_0$ such that 
  $$ m_{j_0} e^{k \Re z_{j_0}}= \max\limits_j m_je^{k \Re z_j}.$$
  
  Now, consider  the analytic in $\mathbb{C} \setminus \cup (3D_j)$ function $f=(u_x - i u_y) e^{kz}$. 
  If $|u(z)| \leq e^{-L|z|}$ in $\mathbb{C} \setminus \cup (D_j)$, then 
  $$|\nabla u(z)| \leq C \sup_{B(z,1)}| u| \leq Ce^{-L(|z|-1)} \quad \text{ for  } z \in \mathbb{C} \setminus \cup (2D_j)$$
   and 
   $f(z) \to 0$ as $z \to \infty$, $z \in \mathbb{C} \setminus \cup (2D_j)$.
   So, by the maximum principle, there exists $j_1$ such that 
   $$ \max\limits_{\mathbb{C} \setminus \cup (3D_j)}|f|= \max\limits_{\partial 3D_{j_1}} |f| \leq Am_{j_1} e^{k\Re z_{j_1}} \leq Am_{j_0} e^{k\Re z_{j_0}},$$ 
   whence $|\nabla u| \leq Am_{j_0} e^{-k\Re(z-z_{j_0})}$ in $\mathbb{C} \setminus \cup (3D_j)$. We may assume that
   $m_{j_0}\neq 0$, otherwise $u$ is constant and therefore zero.
   
   Now, consider the ray $\{z_{j_0} + y: y \in (0,+\infty) \}$.     
   There are two possibilities:
   \begin{itemize}
   \item[(i)] The ray goes to $\infty$ without hitting any other disks $(3D_j)$.
   Then for any $y>0$, 
   $$|u(z_{j_0}+y) - u(z_{j_0})| \leq \int_{0}^{\infty}|\nabla u(z_{j_0}+t)| dt \leq \int_{0}^{\infty} Am_{j_0} e^{-kt}dt = \frac A k m_{j_0}.   $$  
   Since $|u(z_{j_0})| \geq m_{j_0}$, wee see that $|u|$ stays bounded from below by 
   $(1- \frac A k)m_{j_0}$ on the ray. If $k>A$, this contradicts the decay assumption.
   \newpage
   \item[(ii)] The ray hits another disk $3D_j$
   
  \includegraphics[width=0.9\textwidth]{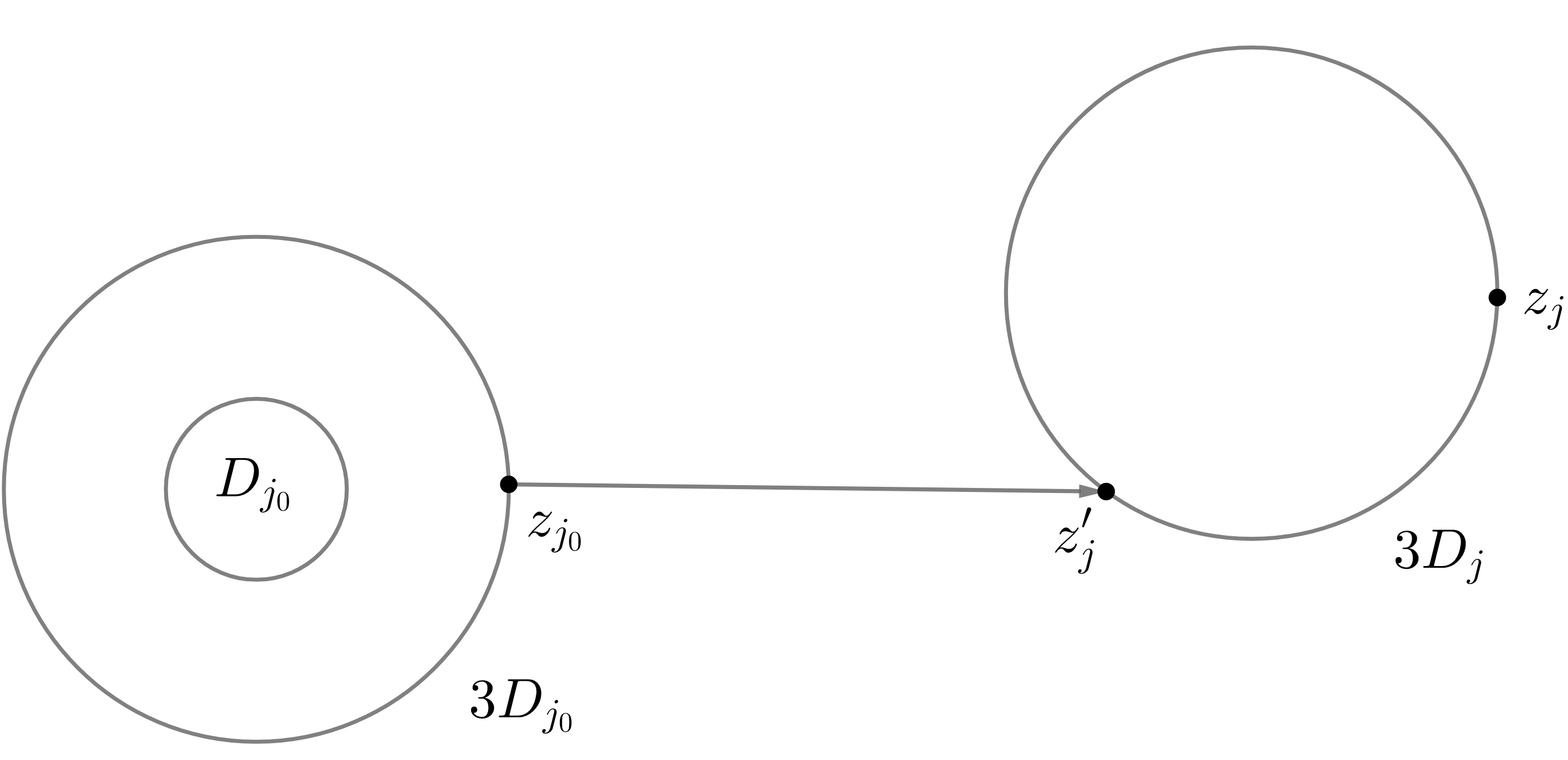} 

  \noindent at some point $z_{j}'= z_{j_0}+y$.
  Then we still have $|u(z_j')|\geq (1-\frac A k )m_{j_0}$
   and, due to the fact that the disks are separated,
   $$ \Re (z_{j}' - z_{j_0} )=|z_{j}' - z_{j_0}| \geq 1.$$  
   Hence $$m_j e^{k\Re z_j } \geq m_j e^{k \Re z_j'} \geq \frac{|u(z_j')|}{A} e^{k(\Re z_{j_0} +1)} \geq$$
   $$\geq \frac{1}{A}(1 - \frac{A}{k})e^k m_{j_0} e^{\Re z_{j_0}} > m_{j_0} e^{\Re z_{j_0}} $$
   as soon as $k>2A$, which contradicts the choice of $j_0$.
   
   \end{itemize}
   \noindent This proves the theorem with any $L>2A$.

 \end{proof}
 
 Now we formulate and prove the harmonic counterpart of the main local theorem.
 
 \begin{theorem} \label{local 3}
 Let $D_j$ be a collection of 100 -- separated unit disks on $\mathbb{R}^2=\mathbb{C}$ such that 
 $0 \notin \cup 3D_j$. Let $R> 10^4$, $0<r\leq R/4$. 
 Consider any harmonic function $u$ in $B(0,R) \setminus \cup D_j$ such that $u$ does not change sign in
 $(5D_j\setminus D_j) \cap B(0,R)$ for every $j$. Assume  that 
 $$\sup \limits _{B(0,R-R/64) \setminus 
 \cup 3D_j}|u| \geq e^{- N} \sup \limits_{B(0,R) \setminus 
 \cup 3D_j}|u|.$$
 Then 
 \begin{equation}\label{eq:*}
  \sup \limits _{B(0,r) \setminus 
 \cup 3D_j}|u| \geq  \left(\frac r R \right)^{ C (R + N) } \sup \limits_{B(0,R) \setminus 
 \cup 3D_j}|u|,
 \end{equation}
with some absolute constant $C>0$.
 \end{theorem}
 
 \begin{proof}

 WLOG, $\sup \limits _{B(0,R-R/64) \setminus 
 \cup 3D_j}|u|=1$. Fix $k=[C(N+R)]$ with sufficiently large $C>0$ and assume that
 $$ \sup \limits _{B(0,r) \setminus 
 \cup 3D_j}|u| \leq \left(\frac r R \right)^{3k}.$$
 
 Consider the
domain $$\Omega:=\{r/2<|z|< R - 1 \} \setminus  \cup (3D_j).$$ 
Let $W_1$ be the connected component of $\partial \Omega$ that intersects $\partial B(0,r/2)$.
Note that each point of $W_1$ is either on  $\partial B(0,r/2)$ or lies on some $\partial 3D_j$ that intersects $\partial B(0,r/2)$.\\
\textbf{Estimate on $W_1$.}
Recall that if $5D_j \subset B(0,R)$, we have
 \begin{enumerate}
 \item  $\max\limits_{\partial3 D_j }|u| \leq A \min\limits_{\partial 3D_j}|u|$,
 \item  $\max\limits_{\partial 3D_j}|\nabla u| \leq A \min\limits_{\partial 3D_j}|u|$.
 \end{enumerate}
Hence on $W_1 \setminus \partial B(0,r/2) $, we have
$$ |u|,|\nabla u| \leq A \sup_{B(0,r)\setminus \cup (3D_j)} |u| \leq A  \left(\frac r R \right)^{3k}. $$
If $x \in \partial B(0,r/2) \setminus \cup (3D_j)$, then either $x \in 4D_j$ for some $j$ or $B(x, min(1,r/2)) \subset B(0,r) \setminus \cup (3D_j) $. In the first case $u$ does not change sign in $B(x,1)$ and 
$$|\nabla u(x)| \leq A|u(x)| \leq A\sup_{B(0,r)\setminus \cup (3D_j)} |u|\leq A\left(\frac r R \right)^{3k}.$$
In the second case, we have 
$$ |\nabla u(x)| \leq\frac{A}{\min(1,r/2)}\sup_{B(0,r)\setminus \cup (3D_j)}|u| \leq \frac{A}{\min(1,r/2)} \left(\frac r R \right)^{3k}. $$
 Thus in all cases, if  $C$ in the definition of $k$ is large enough, we have
 $$ \max_{W_1}|u|, \max_{W_1}|\nabla u| \leq \frac{A}{\min(1,r/2)} \left(\frac r R \right)^{3k} \leq \left(\frac r R \right)^{2k} \quad $$
 because  $$\left( \frac R r \right)^{k} \geq 4^k > A$$
 and
 $$\left( \frac R r \right)^{k} \geq 4^{k-1}  \frac R r  \geq \frac{2A}{r}.$$
 
 Let $W_2$ be the connected component of $\partial \Omega$ that intersects $\partial B(0,R-1)$.
 Note that each point of $W_2$ is either on  $\partial B(0,R-1)$ or lies on some $\partial 3D_j$ that intersects $\partial B(0,R-1)$.\\
 \textbf{Estimate on $W_2$.} Any point $x\in \overline{B(0,R-1)}\setminus \cup (3D_j)$ is either in $4D_j$ for some $j$
 or $x\in \overline{B(0,R-1)}\setminus \cup (4D_j)$. In the first case $u$ does not change sign in $B(x,1)$ and therefore
 $$|\nabla u(x)|\leq A|u(x)| \leq A e^N.$$
 In the second case $B(x,1) \subset B(0,R)\setminus \cup (3D_j)$ and $|\nabla u(x)| \leq A e^N$.
 Thus 
 $$ \max_{W_2}|u|, \max_{W_2}|\nabla u| \leq A e^N.$$
 Note also that 
 $$ W_2 \subset \overline{B(0,R-1)}\setminus B(0,R-7).$$

 Now, consider the analytic in $\Omega$ function 
  $$f(z)=\frac{u_x-iu_y}{z^k}, \quad |f(z)|=\frac{|\nabla u(z)|}{|z|^k}.$$
  Since $$\sup\limits_{B(0,R-R/64)\setminus \cup (3D_j)}|u|=1> \sup_{B(0,r/2)\setminus \cup 3D_j}|u|,$$  $$\max\limits_{W_1}  |u| \leq \left( \frac r R \right)^{2k} < \frac{1}{2},$$
 and since any point in $$\Omega_1=B(0,R-R/64)\setminus \cup (3D_j)$$ can be connected with $W_1$ by a curve of length at most $4R$ within $\Omega_1$, we must
 have $$\sup_\Omega|\nabla u| \geq \sup_{\Omega_1}|\nabla u| \geq \frac{1}{8R}$$ and
 $$ \sup_{\Omega_1} |f|\geq \left( R-R/64\right)^{-k} \sup_{\Omega_1}|\nabla u| \geq \frac{1}{8R} \left( R-R/64 \right)^{-k}.$$
However 
$$ \max_{W_1}|f|\leq \max_{W_1}|\nabla u |\left(\frac{2}{r} \right)^k \leq \left(\frac{r}{R} \right)^{2k}\left(\frac{2}{r} \right)^k= \left(\frac{2r}{R} \right)^{k} R^{-k} \leq $$ $$ \leq 2^{-k} R^{-k} < \frac{1}{8R} R^{-k} < \sup_{\Omega_1} |f|\leq \sup_{\Omega} |f| $$
and
$$ \max_{W_2}|f|\leq A e^N \frac{1}{(R-7)^k} = A e^N \left( \frac{R-R/64}{R-7}\right)^k \left( R-R/64\right)^{-k}\leq 
$$ $$\leq A e^N \left( \frac{126}{127}\right)^k \left( R-R/64\right)^{-k} < \frac{1}{8R} \left( R-R/64\right)^{-k} \leq \sup_{\Omega} |f| $$
if $R-7> \frac{127}{128}R $ and $C$ in the definition of $k$ is large enough. 
By the maximum principle for holomorphic functions $\sup_\Omega |f|$ is achieved on $\partial 3D_j$ for some 
$3D_j\subset B(0,R-1)\setminus \overline{B(0,r/2)}$. 

For every disk $D_j$ with $3D_j\subset B(0,R-1)\setminus \overline{B(0,r/2)}$, consider the point $z_j$ on $\partial3D_j$ closest to the origin. All $3D_j$ that are not in the annulus  $B(0,R-1)\setminus \overline{B(0,r/2)}$ will not be considered further. Put $m_j=\min_{\partial3D_j}|u|$. Let $j_0$ be the index such that $$\frac{m_{j_0}}{|z_{j_0}|^k} = \max_j \frac{m_{j}}{|z_{j}|^k}.$$
If $\sup_\Omega|f|$ is achieved on $\partial3D_{j_1}$, then for $x\in \Omega$,
$$ \frac{|\nabla u(x)|}{|x|^k}\leq \frac{1}{|z_{j_1}|^k} \max_{\partial3D_{j_1}}|\nabla u|\leq A \frac{m_{j_1}}{|z_{j_1}|^k} \leq A \frac{m_{j_0}}{|z_{j_0}|^k}.$$
So we conclude that 
$$|\nabla u(x)| \leq \left( \frac{|x|}{|z_{j_0}|} \right)^k A m_{j_0}.$$
Recalling that $\frac{1}{8R} \leq \sup_\Omega|\nabla u|$, we get
$$ \frac{1}{8R} \leq \left(\frac{2R}{r}\right)^k Am_{j_0},$$
so 
$$ |u(z_{j_0})| \geq m_{j_0} \geq \frac{1}{8AR} \left(\frac{r}{2R}\right)^k.$$

Now, let $(sz_{j_0},z_{j_0})$ be the longest subinterval of the radius $[0,z_{j_0}]$ starting at $z_{j_0}$ that is contained in $\Omega$. We have 
$$|u(sz_{j_0})| \geq |u(z_{j_0})| - |z_{j_0}| \int_{s}^{1}|\nabla u(t z_{j_0})| dt \geq |u(z_{j_0})| - |z_{j_0}|   \int_{0}^{1}  Am_{j_0} t^k  dt \geq $$
$$ \geq  m_{j_0}(1-\frac{AR}{k+1}) \geq \frac{m_{j_0}}{2}$$
if the constant $C$ in the definition of $k$ is large enough.
Note that $$\left(\frac{R}{r}\right)^{k} \geq 4^{k} > 16 AR \cdot 2^k.$$
Hence $$\frac{m_{j_0}}{2} \geq \frac{1}{16AR} \left(\frac{r}{2R}\right)^k > \left(\frac{r}{R}\right)^{2k}$$ 
and the point $sz_{j_0}$ cannot belong to $W_1$, whence it belongs to some $\partial 3D_j$ with $3D_j \subset B(0,R-1)\setminus \overline{B(0,r/2)}$.
Then
$$ \frac{m_j}{|z_j|^k} \geq \frac{|u(sz_{j_0})|}{A|sz_{j_0}|^k} \geq \frac{1}{2As^k} \frac{m_{j_0}}{|z_{j_0}|^k}.$$
It remains to notice that, since the distance from $3D_j$ to $3D_{j_0}$ is at least $96$, we have 
$$ s^k \leq (1-96/R)^k< \frac{1}{2A} $$
if the constant $C$ in the definition of $k$ is large enough.
But then $\frac{m_j}{|z_j|^k} > \frac{m_{j_0}}{|z_{j_0}|^k}$, which contradicts  the choice of $j_0$.

 \end{proof}

\section{Appendix. }

\subsection{The toy problem for harmonic functions in higher dimensions: a proof with extra logarithm.}

Here we present another proof of a slightly worse bound for the toy problem for harmonic functions in a punctured domain. However this proof works in higher dimensions.
We will denote by $B_R$ the ball in $\mathbb{R}^n$ with center at $0$ and of radius $R$.\\
\textbf{Toy problem with extra logarithm.}
Let $D_j$ be a collection of unit, $100$ -- separated balls on the plane and let $R>100$.
Then for any harmonic function $h$  in $B_R \setminus \cup D_j$ such that $h$ does not change sign in each $B_R\cap 5D_j\setminus D_j$,
we have 
$$ \int\limits_{B_{R}\setminus(B_{R/2} \bigcup (\cup 3D_j))} h^2 \geq \exp(-CR\log R) \int\limits_{B_{R/2}\setminus \cup 3D_j} h^2,$$
 where $C$ is an absolute positive constant.
 
 \noindent This inequality implies that Theorem \ref{thm: toy} holds in higher dimensions if we assume that
 $|u(z)| \leq e^{-L|z|\log|z|}$ for sufficiently large $L$.
 
\begin{proof}
The proof is based on the Carleman inequality with log linear weight.
Most of Carleman inequalities require strict log convexity-type properties of the weight.
The next inequality is an exception:
\begin{equation} \label{eq: Carleman}
 \int_{B_R} |{\Delta} u|^2 e^{kx_1} \geq \frac{ck^2}{R^2} \int_{B_R} u^2 e^{kx_1}
\end{equation}
for any $u\in C^2_0(B_R)$. The inequality is not difficult to prove.
Let $v=ue^{kx_1/2}$, then 
$$ e^{kx_1/2} \Delta u = \Delta v - k v_{x_1} + \frac{k^2}{4} v$$
and 
$$\int_{B_R} |{\Delta} u|^2 e^{kx_1}= \int_{B_R} |\Delta v + \frac{k^2}{4} v|^2  +  \int_{B_R} |k v_{x_1}|^2 - \int_{B_R} 2 (\Delta v + \frac{k^2}{4} v)  kv_{x_1}. $$
Note that 
$$\int_{B_R} 2  v v_{x_1} =\int_{B_R} \frac{\partial}{\partial x_1}v^2=0.$$
Integrating by parts, we see that
$$-\int_{B_R}  \Delta v_{x_1} v=\int_{B_R}  \Delta v v_{x_1} = \int_{B_R}  v \Delta v_{x_1}$$
and therefore $$\int_{B_R}  \Delta v v_{x_1} = 0.$$
Hence
$$\int_{B_R} |{\Delta} u|^2 e^{kx_1}= \int_{B_R} |\Delta v + \frac{k^2}{4} v|^2  +  \int_{B_R} |k v_{x_1}|^2  \geq \int_{B_R} |k v_{x_1}|^2 \geq $$
(by Poincare's inequailty)
$$ \geq \frac{\pi^2}{4}\frac{k^2}{ R^2} \int_{B_R}  v^2= c\frac{k^2}{ R^2} \int_{B_R} u^2 e^{kx_1}.$$
So we proved \eqref{eq: Carleman} and would like to apply it for the harmonic function $h$. However $h$ is not in 
$C^2_0(B_{R})$ and inequality \eqref{eq: Carleman} should be applied to $$u=h\eta,$$
where $\eta$ is a positive, $C^2$-smooth cut-off function:
\begin{itemize}
\item $\eta=0$ in each $2D_j$ and in $\{x: |x|>R-11\}$,
\item $\eta=1$ in $B_{\frac{3}{4}R \setminus \cup 3D_j}$,
\item the function $\eta$, as well as its  first and second derivatives are bounded by a numerical constant.
\end{itemize}

We will choose the parameter $k$ later. For now we have
$$\int_{B_{\frac{3}{4}R}\setminus\cup 3D_j} |{\Delta} h|^2 e^{kx_1} + \textup{``cut-off integrals"} \geq \frac{ck^2}{R^2} \int_{B_{\frac{3}{4}R}\setminus \cup 3D_j} h^2 e^{kx_1} =: \textup{RHS}.$$
It is good that $\Delta h=0$, so only the cut-off integrals are left on the left-hand side.
There are two kinds of cut-off integrals:
$$\textup{I}= \sum\limits_{5D_j \subset B_{\frac{3R}{4}}}  \int_{3D_j\setminus 2D_j} \textup{``cut-off terms"}$$
and 
$$ \textup{II}= \int_{B_{R-11}\setminus B_{(\frac{3R}{4}-10)}} \textup{``cut-off terms"}$$
where 
$$ | \textup{``cut-off terms"}| \lesssim (h^2+|\nabla h|^2)e^{kx_1}$$
(recall that $\Delta h =0$).
Note that $$\int_{3D_j\setminus 2D_j}  e^{kx_1} \lesssim  e^{-k/2}\int_{4D_j\setminus 3D_j}  e^{kx_1}$$
because $4D_j\setminus 3D_j$ contains an open disk of radius $\frac{1}{4}$, where the function $e^{kx_1}$
is pointwise bigger than $e^{k/2} \cdot \sup_{3D_j\setminus 2D_j}  e^{kx_1}$.
Now, assuming $5D_j \subset B_R$ we will use the sign condition in $5D_j\setminus D_j$. By the Harnack inequality and the Cauchy estimate
we know that there is a constant $a_j\geq 0$ such that 
$$ |h| \asymp a_j \textup{ and } |\nabla h| \lesssim a_j \textup{ in } 4D_j\setminus 2D_j.$$  
So
$$\int_{3D_j\setminus 2D_j} \textup{``cut-off terms"} \lesssim
 a_j^2 \int_{3D_j\setminus 2D_j}  e^{kx_1} \lesssim a_j^2 e^{-k/2} \int_{4D_j\setminus 3D_j}  e^{kx_1} \lesssim $$
 $$ \lesssim e^{-k/2} \int_{4D_j\setminus 3D_j} h^2 e^{kx_1}.$$
Hence $$ \textup{I} \lesssim e^{-k/2}\int_{B_{\frac{3}{4}R}\setminus \cup 3D_j} h^2 e^{kx_1}.$$
Note that
$$\textup{RHS}= \frac{ck^2}{R^2} \int_{B_{\frac{3}{4}R}\setminus \cup 3D_j} h^2 e^{kx_1} > 2 \textup{I}$$
if $$ k^2/{R^2} \gg e^{-k/2}.$$
We make the choice  $$ k = C \log R$$
and it yields 
$$ \sup_{B_{R}} e^{kx_1} \leq e^{CR\log R}.$$
Since $$\textup{I}+ \textup{II} \geq \textup{RHS} \quad \textup{and} \quad I\leq \frac{1}{2}\textup{RHS},$$
we have 
$$  \int_{B_{R-11}\setminus B_{(\frac{3R}{4}-10)}} \textup{``cut-off terms"} = \textup{II} \geq  \frac{1}{2}\textup{RHS} \asymp k^2/{R^2} \int_{B_{\frac{3}{4}R}\setminus \cup 3D_j} h^2 e^{kx_1} \geq $$
$$ \geq \exp(-CR\log R)\int_{B_{\frac{1}{2}R}\setminus \cup 3D_j} h^2.$$
If $5D_j \subset B_R$, then $\int_{3D_j\setminus 2D_j} h^2 \asymp \int_{4D_j\setminus 3D_j} h^2$, whence
$$ \int\limits_{B_{R}\setminus(B_{R/2} \bigcup (\cup 3D_j))} h^2 
\geq c_1  \int\limits_{B_{R-10}\setminus(B_{R/2} \bigcup(\cup 2D_j))} h^2 \geq $$
(by Cauchy estimate)
 $$\geq c_2  \int\limits_{B_{R-11}\setminus(B_{R/2} \bigcup(\cup 3D_j))} \! \! \! (h^2 +|\nabla h|^2)$$
and therefore
$$ \sup_{B_{R}} e^{kx_1}  \int\limits_{B_{R}\setminus(B_{R/2} \bigcup (\cup 3D_j))} h^2  \geq  c_2 \int\limits_{B_{R-11}\setminus(B_{R/2} \bigcup(\cup 3D_j))} (h^2 +|\nabla h|^2) e^{kx_1} \geq c_3 \textup{II}. $$
Thus 
$$ \int\limits_{B_{R}\setminus(B_{R/2} \bigcup( \cup 3D_j))} h^2  \geq \exp(-C'R\log R) \int\limits_{B_{R/2}\setminus \cup 3D_j} h^2.$$
\end{proof} 

\noindent \textbf{Deduction of Theorem \ref{exercise 1} from Theorem \ref{local 3}}.
We may assume that $$\sup\limits_{B(0,R')\setminus \cup 3D_j}|h|=\sup\limits_{B(0,R'/8)\setminus \cup 3D_j}|h|,$$ otherwise the statement is trivial.
Consider any point $x$ on $\partial B(0,R'/4)\setminus \cup 3D_j$.

Note that $ B(0,R'/8) \subset B(x,3R'/8)$ and $B(x,3R'/4) \subset B(0,R')$. Hence
$$ \sup\limits_{B(x,3R'/8)\setminus \cup 3D_j}|h|=\sup\limits_{B(x,3R'/4)\setminus \cup 3D_j}|h|.$$
Applying Theorem \ref{local 3} for the disk with center at $x$ (in place of $0$) of radius $R=3R'/4$ and $N=0$, we obtain the bound
$$\sup\limits_{\{R'/8<|z|< R' \}\setminus \cup 3D_j}|h|\geq \sup\limits_{B(x,R'/8)\setminus \cup 3D_j}|h| \geq e^{-CR'}\sup\limits_{B(0,R')\setminus \cup 3D_j}|h|.$$

\subsection{Sketches of general elliptic theory.}

\begin{fact} \label{fact1}
 Denote by $E(z)= 
 \frac{1}{2\pi} \log|z|$ the fundamental solution of the Laplace operator on the plane
 in the sense that for every $C^\infty$ compactly supported function $h$, we have
 $$ h = E * \Delta h.$$
 
\end{fact}

\begin{fact} \label{fact2}
 Let $\Omega$ be a bounded  open set and $g \in L^1(\Omega)$. Put $f= E* g$.
 Then 
 \begin{enumerate}
 \item $f\in L^p(\Omega)$ for all $p \geq 1$.
 \item $\Delta f =g $ in the sense that for every $h \in C_0^\infty(\Omega)$, we have 
 $$\int_{\Omega} f \Delta h = \int_{\Omega} gh.$$
 \end{enumerate}

\end{fact}

\noindent \textbf{Agreement}. Writing $E*g$ we assume that $g$ is extended to $\mathbb{R}^2\setminus \Omega$ by zero.
\begin{proof} \quad

\begin{enumerate}
 \item
Let $D=\textup{diam }\Omega$. Let $E_D= \mathbbm{1}_{B(0,D)} E$. Then in $\Omega$ we have $f=g* E_D$. 
Since $E_D \in L^p(\mathbb{R}^2)$ for all $p\geq 1$ and $g\in L^1$, the result follows from Young's convolution inequality.

\item 
We have 
$$\int f\Delta h = \int (E*g)\Delta h =\int \limits_{\Omega \times \Omega} E(z-\zeta)g(\zeta) \Delta h(z)dm_2(z)dm_2(\zeta)=$$
$$= \int\limits_{\Omega} g(\zeta) \left[ \int\limits_{\Omega} E(z-\zeta)\Delta h(z) dm_2(z) \right] dm_2(\zeta)= \int\limits_\Omega gh. $$
\end{enumerate}

\end{proof}

\begin{fact} \label{fact3}
 Let $V \in L^\infty(\Omega)$, $u\in L^1_{loc}(\Omega)$ and $\Delta u+ Vu=0$ in $\Omega$
 in the sense that for every $h\in C^{\infty}_{0}(\Omega)$ we have 
$$\int\limits_{\Omega}u\Delta h + \int\limits_{\Omega}Vuh=0.$$
Then $u \in L^p_{loc}(\Omega)$ for every $ p\geq 1$. 
 
\end{fact}

\begin{proof}
Passing to a smaller bounded domain $\Omega'$, we may assume that $u\in L^1(\Omega)$, $\Omega$ is bounded.
Consider $f= E * (Vu)$. By Fact \ref{fact2}, $f \in L^p(\Omega)$ for all $p \geq 1$. Note that $u-f \in L^1(\Omega)$ and $\Delta(u-f)=0$ in the sense of distributions. Hence, by Weyl's lemma, $u-f$ is harmonic in $\Omega$,
so $u-f \in L^p_{loc}(\Omega)$ and therefore $u \in L^p_{loc}(\Omega) $ too.
\end{proof}

\begin{fact} \label{fact4}
 Let $g \in L^p(\Omega)$ with $p>2$ and
 let $\Omega$ be bounded. Then $g * E \in C^1(\Omega)$.
 
\end{fact}

\begin{proof}
 $$E(z+t)-E(z)= \frac{1}{2\pi}(\log|z+t|-\log|z|)= \frac{1}{2\pi} \log\left|1+\frac{t}{z}\right| =$$
 
 $$= \frac{1}{2\pi} 
\Re \frac{t}{z} + O\left(
\begin{cases}
\frac{|t|^2}{|z|^2}, \quad \frac{1}{2} \geq \frac{|t|}{|z|} \\
\frac{|t|}{|z|}+|\log \bigl|1+\frac{t}{z}|\bigr|,\quad  \frac{1}{2} \leq \frac{|t|}{|z|}
\end{cases}
\right).
$$
Define $$ W(\zeta) =\begin{cases}
\frac{1}{|\zeta|^2}, \quad \zeta > 2  \\
\frac{1}{|\zeta|}+\bigl|\log|1+\frac{1}{\zeta}|\bigr|,\quad   \zeta \leq 2.
\end{cases}$$

 Taking the convolution and applying Holder's inequality, we have
 
 $$ (g*E)(z+t)- (g*E)(z) =  \left[ g*  \frac{1}{2\pi} 
\Re \frac{t}{\cdot} \right](z) + O\left(\|g\|_p \|W(\cdot/t)\|_q\right) = $$
$$= \Re\left(t  \left[g*\frac{1}{2\pi \cdot} \right] (z)\right) + O\left(\|W(\cdot/t)\|_q\right), \quad \text{where } 1/p+1/q=1.$$
Since $g*\frac{1}{2\pi \cdot} \in C(\Omega)$,  the first term (as a function of $t \in \mathbb{C}=\mathbb{R}^2$) is a linear operator
from $\mathbb{R}^2$ to $\mathbb{R}$, which depends continuously on $z$.
It is enough to show that $\|W\|_q< \infty$ because $$\|W(\cdot/t)\|_q =   \|W\|_q |t|^{2/q} = \|W\|_q \,o(t) \quad \text{ as } t \to 0$$
 ($1<q<2$ if $p>2$). 
Indeed,
$$\int|W|^q  \lesssim \int_{|\zeta|>2} \frac{1}{|\zeta|^{2q}} + \int_{|\zeta|\leq 2} \left( \frac{1}{|\zeta|^q}+ \left|\log\Bigl|1+\frac{1}{\zeta}\Bigr|\right|^{q} \right)< \infty.$$

\end{proof}

\begin{fact} \label{fact5}
 Let $V \in L^\infty(\Omega)$. If $\Delta u + Vu=0$ in $\Omega$ in the sense of distributions and 
$u \in L^1_{loc}(\Omega)$, then $u \in C^1(\Omega)$. 
 
\end{fact}

\begin{proof}
By Fact \ref{fact3}, $u \in L^p_{loc}(\Omega)$ with $p > 2$. 
Again passing to a subdomain, if necessary, we may assume that $\Omega$ is bounded and
$u \in L^p(\Omega)$. Consider  $f= E*(Vu)$. Since $Vu \in L^p(\Omega)$, $f \in C^1(\Omega)$.
However $u-f$ is harmonic. Hence $u \in C^1(\Omega)$.  
\end{proof}

\begin{lemma} \label{lem: solving}
Let $\Omega$ be a bounded domain with Poincare constant smaller than 1.
Then for any $v\in L^\infty(\Omega)$, we can find a solution $u \in W^{1,2}_0(\Omega)$ to 
$\Delta u = v$ in the sense of distributions
such that
$$\|u\|_{W^{1,2}_0(\Omega)} \leq 4 \|v\|_2.$$

\end{lemma}
\begin{remark} Note that if  $u \in W^{1,2}_{loc}(\Omega)$ and $h \in C^\infty_0(\Omega)$, then
$$ \int_\Omega \nabla u \nabla h = -\int_\Omega u \Delta h.$$
So $u \in W^{1,2}_{loc}(\Omega)$ is a solution to $\Delta u = v$ in the sense of distributions if and only if 
$$ \int_\Omega \nabla u \nabla h = -\int_\Omega v  h $$
for any $h \in C^\infty_0(\Omega)$.
\end{remark}

\begin{proof}

Consider the functional 
$$ \Phi(u)= \int |\nabla u|^2 + \int vu$$ 
for $u \in W_{0}^{1,2}(\Omega)$.  Integrals in the next few lines will be over the domain $\Omega$. Notice that by Poincare's inequality

$$\Phi(u) \geq \frac{1}{2} \int |\nabla u|^2 + \frac{1}{2} \int |u|^2 - \|v\|_2\|u\|_2 \geq$$
$$\geq \frac{1}{2} \|u\|^2_{W_{0}^{1,2}} - \|v\|_2\|u \|_{W_{0}^{1,2}} \geq -\frac{1}{2}\|v\|_2. $$

 Thus $\Phi(u)$ is bounded from below. Note that $\Phi(u)>0=\Phi(0)$
 as soon as $\|u\|_{W_{0}^{1,2}} > 2\|v\|_2.$
 Let now $u_k\in C_0^{\infty}(\Omega)$ be any minimizing sequence for $\Phi$. Note  that 
 
 $$ \frac{\Phi(u')+ \Phi(u'')}{2} - \Phi\left(\frac{u'+u''}{2}\right) = \frac{1}{4} \int |\nabla(u'-u'')|^2 \geq \frac{1}{8} \|u'- u''\|^2_{W_0^{1,2}}.$$
 
 Hence $u_k$ is a Cauchy sequence in $W_0^{1,2}(\Omega)$, so the limit $u=\lim u_k$ exists and minimizes $\Phi$. 
 Now, take any test function $h \in C_0^\infty(\Omega)$ 
 and consider 
 $$ \Phi(u+th)= \Phi(u) + t \left(2\int \nabla u \nabla h + \int vh \right)
 +t^2 \int |\nabla h|^2.$$
 Since $u$ is a minimizer, we must have 
 $$2\int \nabla u \nabla h + \int vh=0,$$ i.e., 
 $\Delta u = v/2$ in the sense of distributions. 
 Taking $2u$ in place of $u$ we get a solution to $\Delta u = v$  with 
 $$\|u\|_{W_0^{1,2} } \leq 4 \|v\|_2.$$
 \end{proof}

 The next step is to 
 show that $$\| u\|_\infty \leq C\| v\|_\infty$$
 with some absolute constant $C>0$.
 WLOG, we will assume $|v|\leq 1$.

 \subsection{Uniform bound via Di Giorgi method}
Let $\Omega$ be any bounded open set in $\mathbb{R}^2$ with
Poincare constant $k^{2} \leq k_{0}^{2},$ i.e.,
\[
\int u^{2} \leq k^{2} \int|\nabla u|^{2} \text { for all } u \in C_{0}^{\infty}(\Omega).
\]
\textbf{Claim I}: Let $k_0$ be sufficiently small and consider any smooth
 $\varepsilon-$minimizer of
\[
\Phi(u)=\int_{\Omega}|\nabla u|^{2}+\int_{\Omega} vu \quad\left(\|v\|_{\infty} \leq 1\right)
, \text{ i.e.}, \]
$u\in C^\infty_0(\Omega)$ and for any  $\tilde u\in C^\infty_0(\Omega)$, $\Phi(\tilde u) \geq \Phi(u) -\varepsilon$.
Then $u$ satisfies $$\int_{B\cap \Omega} u^{2} \leq Ck^2(k^2+\varepsilon)$$
 for every unit ball $B \subset \mathbb{R}^{2}$.

\begin{proof}

 Let $\varphi(x)$ be a smooth positive radial function such that 
 \begin{itemize}
 \item $\varphi(x)=1$ in $B(0,1)$, 
 \item $\varphi(x) \in (0,1]$,
 \item
 $\varphi(x) \asymp e^{-|x|}$, 
 \item $|\nabla \varphi| \leq  \varphi$.
  \end{itemize}

 Let $ \psi=\varphi^{2}, \text { so }|\nabla \psi| \leqslant 2|\psi|$.
 Applying the Poincare inequality to $\varphi u$, we get
$$
 \int \varphi^{2} u^{2}\leq k^{2} \int|\varphi \nabla u+u \nabla \varphi|^{2}
\leq 2 k^{2}\left(\int \varphi^{2}|\nabla u|^{2}+\int u^{2}|\nabla \varphi|^{2}\right) \leq $$
 $$\leq 2k^{2} \int \varphi^{2}|\nabla u|^{2}+2 k^{2} \int u^{2} \varphi^{2}, \text { whence } $$
 $$ \int \varphi^2 u^2 \leq \frac{2 k^{2}}{1-2 k^{2}} \int \varphi^2|\nabla u|^{2}, \text{ i.e.}, $$
 \begin{equation} \label{eq: *}
  \int u^{2} \psi \leq \frac{2 k^{2}}{1-2 k^{2}} \int |\nabla u|^{2} \psi \leq 4 k^{2} \int |\nabla u|^{2} \psi \quad \text{ if } \quad k_{0} \leq \frac{1}{2}.
 \end{equation}

Now, consider the competitor $\tilde{u}=(1-\psi) u$.
We have
\[
\quad \Phi(\tilde{u})=\int|(1-\psi) \nabla u-u \nabla \psi|^{2}+\int v(1-\psi) u\leq \]
\[\leq \int(1-\psi)^{2}|\nabla u|^{2}+2 \int|\nabla u| | u||\nabla \psi|+\int u^{2}|\nabla \psi|^{2}+\int v(1-\psi) u
\]
$$\leq \int(\nabla u)^{2}+\int v u-\int |\nabla u|^{2}\psi+4 \int |\nabla u| | u| \psi +4 \int u^{2} \psi-\int v \psi u$$
(we used the inequalities $(1-\psi)^{2} \leq 1-\psi, |\nabla \psi| \leqslant 2 \psi, \psi^{2} \leq \psi$).

Since $\Phi(\tilde{u}) \geqslant \Phi(u)-\varepsilon,$ we must have
\[
\int|\nabla u|^{2} \psi \leq 4\left(\int|\nabla u| |u| \psi+\int u^{2} \psi\right)+\varepsilon+\int | v \psi u |.
\]
However 
\[
 \int u^{2} \psi \leqslant 4 k^{2} \int |\nabla u|^{2} \psi \quad \text { by \eqref{eq: *}, }
\]
and
$$\int|\nabla u| |u| \psi \leq \sqrt{\int u^{2} \psi} \sqrt{\int | \nabla u|^{2} \psi} \leqslant 2 k \int|\nabla u|^{2} \psi.$$
So for sufficiently small $k_0$,
$$ 4\left(\int|\nabla u| |u| \psi+\int u^{2} \psi\right) \leq \frac{1}{2} \int|\nabla u|^{2} \psi.$$
Hence
\[
\int|\nabla u|^{2} \psi\leq 2\varepsilon+ 2\int|v  u \psi |
\]
\[
\leq 2\left(\varepsilon+\sqrt{\int \psi} \sqrt{\int u^{2} \psi} \, \, \right) \leqslant C\left(\varepsilon+ k \cdot \sqrt{\int |\nabla u|^{2} \psi} \, \,\right).
\]
If the first term dominates, then
$\int|\nabla u|^{2} \psi \leqslant C \varepsilon $. Otherwise $ \int|\nabla u|^{2} \psi \leqslant C k^{2}$.

By \eqref{eq: *} it follows  that $$\int u^{2} \psi \leq C k^{2}\left(k^{2}+\varepsilon\right).$$

\end{proof}
Note that we did not care in Claim I where the Poincare constant came from and what was special about the geometry of $\Omega$ that made it small. The next lemma gives a simple bound for the Poincare constant of ``thin" domains.

\begin{lemma}\label{le: thin}
Assume $\Omega$ is open and $m_2(\Omega \cap Q)\leq c < 1$ for all unit squares $Q\subset \mathbb{R}^2$. Then the Poincare constant of $\Omega$ is at most $2 + \frac{2}{1-c}$.
\end{lemma}
\begin{proof}

Let $f \in C_0^{\infty}(\Omega)$. 
Extend $f$ by zero outside $\Omega$. It is sufficient to show that
if $Q$ is a unit square, then 
$$ \int_Q | f|^2 \leq \left(2 + \frac{2}{1-c}\right)\int_Q | \nabla f|^2.$$
By tiling the plane with unit squares, it implies
$$ \int_\Omega | f|^2 \leq \left(2 + \frac{2}{1-c}\right) \int_\Omega | \nabla f|^2.$$
Let $Q=[0,1]^2$,
 $$I_x=((x,y): y\in [0,1]) \quad \text{ and } \quad I^y=((x,y): x\in [0,1]).$$
Let $X_0$ be the set of $x\in [0,1]$ such that $I_x$ contains a zero of $f$. Then 
for every $x\in X_0$, we have
$$\max_{I_x} |f| \leq \int_{I_x} |\nabla f|$$
and
$$\int_{I_x} f^2 \leq \max_{I_x} f^2 \leq \left(\int_{I_x} |\nabla f|\right)^2 \leq \int_{I_x} |\nabla f|^2.$$
Hence 
$$\int_{X_0 \times [0,1]} f^2 \leq \int_Q |\nabla f|^2.$$
The set $X_0$ has Lebesgue measure at least $1-c$, whence
there is $x_0 \in X_0$ such that 
$$ \int_{I_{x_0}} f^2 \leq \frac{1}{1-c} \int_Q |\nabla f|^2. $$
\textbf{Claim.} Let $I$ be a unit interval and let $z$ be any point in $I$.
Then $$\int_I f^2 \leq 2|f(z)|^2 + 2 \int_I |\nabla f|^2.$$
Indeed,
$$\int_I f^2 \leq \max_I f^2 \leq \left(|f(z)| + \int_I |\nabla f|\right)^2 \leq \left(|f(z)| + \sqrt{\int_I |\nabla f|^2}\,\,\right)^2 \leq $$ 
$$\leq 2|f(z)|^2 + 2\int_I |\nabla f|^2.$$
For every $y \in [0,1]$, it yields
$$\int_{I^y} f^2 \leq 2|f(x_0,y)|^2 + 2 \int_{I^y} |\nabla f|^2.$$
Thus 
$$\int_Q f^2 = \int_0^1 \left(\int_{I^y} f^2dx\right)dy \leq 2\int_{I_{x_0}} f^2+2 \int_{Q} |\nabla f|^2 \leq \left(\frac{2}{1-c}+2\right)\int_{Q} |\nabla f|^2.$$
  
\end{proof}
\begin{corollary}\label{cor: thin}
If $$m_2(\Omega \cap Q)\leq k^2 \ll 1$$ 
for any unit square $Q$, then 
the Poincare consant of $\Omega$ is smaller than $Ck^2$. 
\end{corollary}
\begin{proof}
For every square $Q_{2k}$ of size $2k$, we have
$$m_2(\Omega \cap Q_{2k})\leq \frac{1}{4} m_2( Q_{2k}).$$
By  $2k$ rescaling we reduce the problem to Lemma \ref{le: thin}.

\end{proof}
\pagebreak
Now we are almost ready to run the Di Georgi scheme. The only remaining preparatory part is smooth surgery.
Let $u\in C_0^\infty(\Omega)$. Fix $t>0$ (level) and $\delta>0 $ (extremely small number).
Let $\Theta$ be a $C^\infty$-smooth function on $\mathbb{R}$ described by Figure \ref{Theta}.

 \begin{figure}[h!]
   \includegraphics[width=0.7\textwidth]{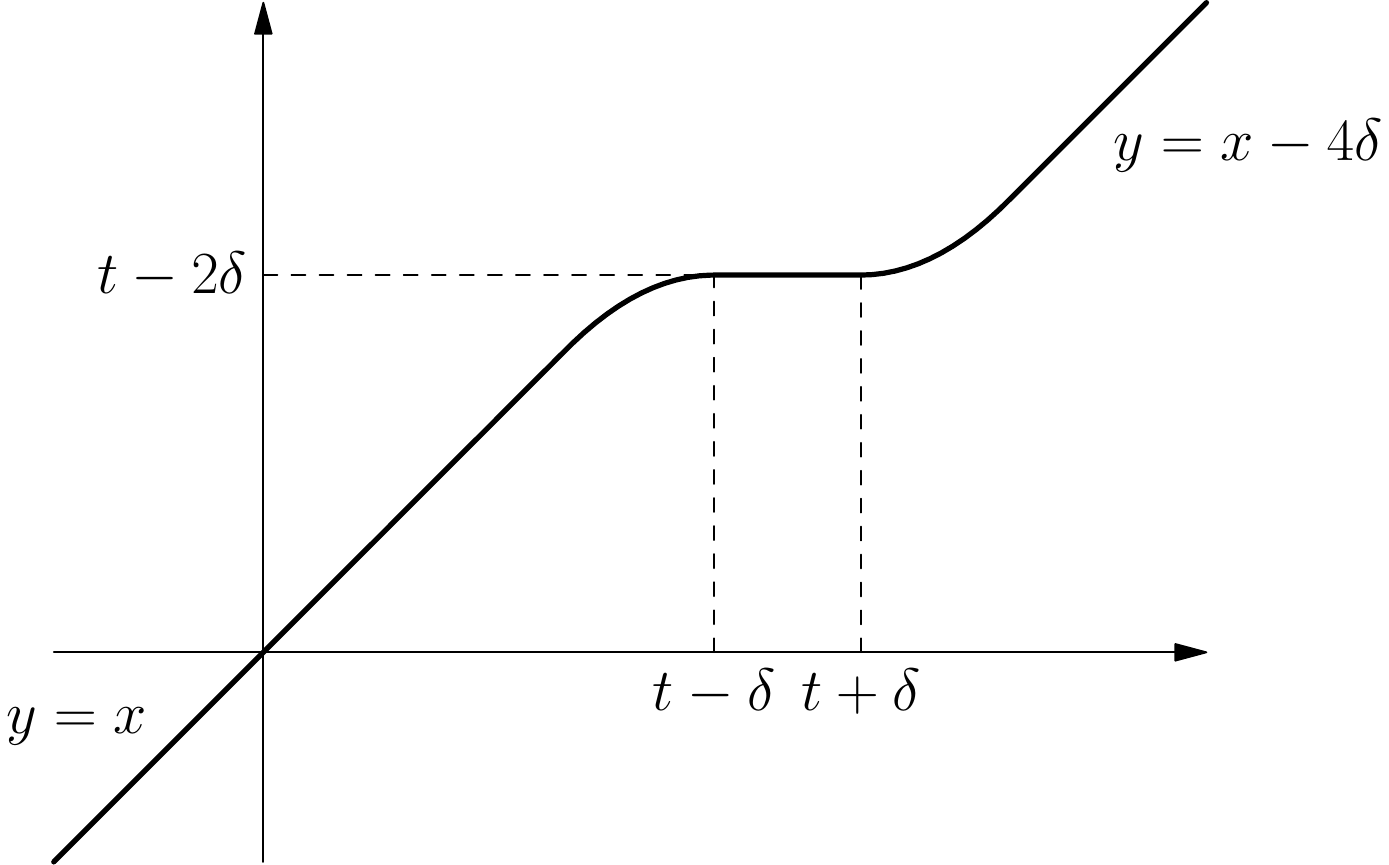}
   \caption{ $y=\Theta(x)$} \label{Theta}
\end{figure}

 The function $\Theta$ has the following properties:
 \begin{itemize}
\item $ 0 \leq  \Theta' \leq 1,$
\item $ \Theta(0)= 0,$
\item $ x-4\delta\leq \Theta(x) \leq x,$
\item $\Theta(x) = t-2\delta \text{ on } (t -\delta, t+\delta).$
 
 \end{itemize}
Let $\tilde u = \Theta \circ u$. Then $|u-\tilde u| \leq 4\delta$ and 
$|\nabla \tilde u| \leq |\nabla u|$ pointwise. Thus, if $u$ was an $\varepsilon$-minimizer,
then $\tilde u$ is an $\varepsilon+ A\delta$-minimizer ($\delta$ is purely qualitative and $A=4\int_\Omega |v|$).
Define $$\Theta_-(x)=
\begin{cases}
\Theta(x), \quad x\leq t+\delta\\
t-2\delta, \quad x\geq t+\delta
\end{cases} \quad \text{and} \quad \Theta_+=\Theta-\Theta_-.  $$
The function $\tilde u$ naturally splits into two smooth compactly supported terms:
$$ \tilde u= \tilde u_-+ \tilde u_+, \text{ where } \tilde u_\pm = \Theta_\pm \circ u. $$
The function $\tilde u_+$ is compactly supported in $\{u>t\}$ ($\text{supp } \tilde u_+ \subset\{u \geq t+\delta \}$)
and $\nabla \tilde u_+$ and $\nabla \tilde u_-$ have disjoint supports.
We may then try to replace $\tilde u_+$ by some smooth competitor $w\in C_0^\infty(\{u>t\})$ and see if the functional 
can drop. Note that 
$$ \Phi(\tilde u_-+w)= \int_\Omega|\nabla \tilde u_-|^2 + \int_{\{u>t\}}|\nabla w|^2 + \int_\Omega v\tilde u_- + \int_{\{u>t\}} v w,$$
so we just need to compare 
$$\int_\Omega|\nabla \tilde u_+|^2  + \int_\Omega v\tilde u_+ \quad \text{with} \quad \int_{\{u>t\}}|\nabla w|^2 + \int_{\{u>t\}} v w.$$
Hence $ \tilde u_+$ is an $(\varepsilon+A\delta)$-minimizer in the new domain $\{u> t \}$.
We shall now fix the initial Poincare constant to be $k_0$ from Claim I. If $u$ is an $\varepsilon$-minimizer, then
$ \int_B u^2 \leq Ck_0^2(k_0^2+\varepsilon)$ for every unit ball $B$, so 
$$m_2(\{u>t_0\} \cap B)\leq \frac{Ck_0^2(k_0^2+\varepsilon)}{t_0^2}.$$

Choose $t_{0}=C' \sqrt{k_{0}}$ with sufficiently large absolute constant $C'$. Then the domain
$\Omega_{1}=\left\{u>t_{0}\right\}$ satisfies $m_2(\Omega_1\cap B) \leq \frac{C}{C'} k_{0}\left(k_{0}^{2}+\varepsilon\right)$ for any unit ball $B$ and, by Corollary \ref{cor: thin},
the Poincare constant of $\Omega_1$ is at most $k^{2}_1:=k_{0} \frac{k_{0}^{2}+\varepsilon}{2}$.
Also $u_{1}=\tilde{u}_{+} \in C_{0}^{\infty}\left(\Omega_{1}\right)$ will be an $\varepsilon_{1}=\varepsilon+A\delta_{0}$-minimizer of $\Phi$ 
 in $C_{0}^{\infty}\left(\Omega_{1}\right)$ where $\delta_{0}>0$ can be chosen arbitrary small. Finally, note
that $u \leq t_{0}+u_{1}+4\delta_{0}$ everywhere in $\Omega$. We can now repeat this construction with $ u_{1}, \Omega_{1}, \varepsilon_1$ instead of $u, \Omega, \varepsilon$ to
get $u_{2}, \Omega_{2}, \varepsilon_{2}$ and so on. We shall get a sequence of domains $\Omega_j$, functions
$u_j \in C_0^{\infty}(\Omega_j)$ and numbers $k_j,t_j,\varepsilon_j, \delta_j>0$ such that 
$$ \Omega_1 \supset \Omega_2 \supset ... \quad, \quad k_j^2= k_{j-1}\frac{k_{j-1}^2+\varepsilon_{j-1}}{2},$$
$$ \varepsilon_j= \varepsilon_{j-1}+ A\delta_{j-1}, \quad t_j=C'\sqrt{k_j},  $$
 $$m_2(\Omega_j\cap B) \leq \frac{C}{C'}k_{j-1}(k_{j-1}^2 +\varepsilon_{j-1}) \quad \text{ for any unit ball } B,  $$
and 
$$ u \leq t_0+t_1+...+t_{l}+u_{l+1}+4\delta_0+4\delta_1+...+4\delta_{l} \quad \text{for any } l\geq 0.$$

In this construction, we can choose $\delta_j>0$ as  small as we want, so putting $\delta_j=\frac{c\varepsilon}{2^j}$ with sufficiently small absolute constant $c>0$, we can guarantee that all $\varepsilon_j \leq 2 \varepsilon$.

Let $l$ be the first index for which $k_l^2<2\varepsilon$. Then (provided that $C' > 8C$, $\varepsilon <1/2$), we also have 
$m_2(\Omega_{l+1} \cap B)<\frac{\varepsilon}{2}$ for every unit ball $B$. For all $j\leq l$, we have
$$k_j^2 \leq k_{j-1}^3,$$ so, if $k_0$ was chosen less than $\frac{1}{4}$, it implies that $k_j^2 \leq 2^{-j-2}$
for $j=0,1,...,l$, whence $t_j \leq C'2^{-\frac{j}{2}-1}$ and 
$$  u \leq t_0+t_1+...+t_{l}+4\delta_0+4\delta_1+...+4\delta_{l} \leq C' \sum_{j\geq 0} 2^{-\frac{j}{2}-1}+\sum_{j\geq 0} 4\delta_j \leq $$
 $$\leq 2C' +\varepsilon \leq 2C'+1=C_0$$
in $\Omega\setminus \Omega_{l+1}$ because $u_{l+1}=0$ in $\Omega\setminus \Omega_{l+1}$.
Thus
$$m_2(\{ u>C_0\}\cap B)<\varepsilon/2.$$
Considering $-u$ instead of $u$, we conclude that also $m_2(\{ u<-C_0\}\cap B)<\varepsilon/2$ and
therefore
$$m_2(\{ |u|>C_0\}\cap B)<\varepsilon$$
for every $\varepsilon$-minimizer $u$. Since the true minimizer $u$ is the limit
of $\varepsilon$-minimizers in $L^2(\Omega)$, we get 
$$ m_2(\{|u|>C_0\}\cap B)=0 \quad \text{ for any unit ball } B,$$
so $m_2(\{|u|>C_0\})=0$. 

\noindent \textbf{Conclusion.} If $|v|\leq 1$ and the Poincare constant of $\Omega$ is not greater than $k_0^2\ll 1$, then the minimizer
of $\int_\Omega |\nabla u|^2 + \int_\Omega vu$ in $W_0^{1,2}(\Omega)$ satisfies 
\begin{equation}
\|u\|_\infty\leq C_0.
\end{equation}

If the Poincare constant of $\Omega$ is $k$, we put $\widetilde \Omega = \frac{k_0}{k}\Omega$, $\widetilde u = \frac{k_0^2}{k^2}u(\frac{k}{k_0}\cdot)$, $\widetilde v = v(\frac{k}{k_0}\cdot)$, so 
$\widetilde \Phi(\widetilde u)= \int_{\widetilde \Omega} |\nabla \widetilde u |^2+ \int_{\widetilde \Omega} \widetilde v \widetilde u = \frac{k_0^4}{k^4}\Phi(u)$. Applying the result that
 was just obtained, we get the final observation.

\begin{lemma}  \label{lem: solving2}
If the Poincare constant of $\Omega$ is $k>0$,
then the minimizer of $\Phi(u)=\int_\Omega |\nabla u|^2 + \int_\Omega vu$ 
in $W_0^{1,2}(\Omega)$ (i.e., the solution to $\Delta u = v/2$) satisfies:
\begin{equation} \label{eq: uniform}
\|u\|_\infty\leq Ck^2\|v\|_\infty,
\end{equation}
 where $C$ is an absolute positive constant.
\end{lemma}
 

\subsection{Other standard facts used in the proof.}

\begin{fact} \label{trivialities 1}
  Let $\Omega$ be a bounded open set, let $u \in W^{1,2}_{0}(\Omega) $ satisfy $\| u\|_\infty \leq 1$. Then there exists a sequence $u_k 
\in C^\infty_0(\Omega)$ with $\|u_k\|_\infty \leq 2$ such that $ u_k \to u, \nabla u_k \to \nabla u $ in $L^2(\Omega)$ and almost everywhere in $\Omega$.
\end{fact}

\begin{proof}
By the definition of $W^{1,2}_0(\Omega)$, we  can find $\tilde u_k 
 \in C^\infty_0(\Omega) $ with $ \tilde u_k \to u, \nabla \tilde u_k \to \nabla u $ in $L^2$.
 Let $\Theta $ be defined by Figure \ref{Theta2}.
 
  \begin{figure}[h!] 
    \includegraphics[width=0.7\textwidth]{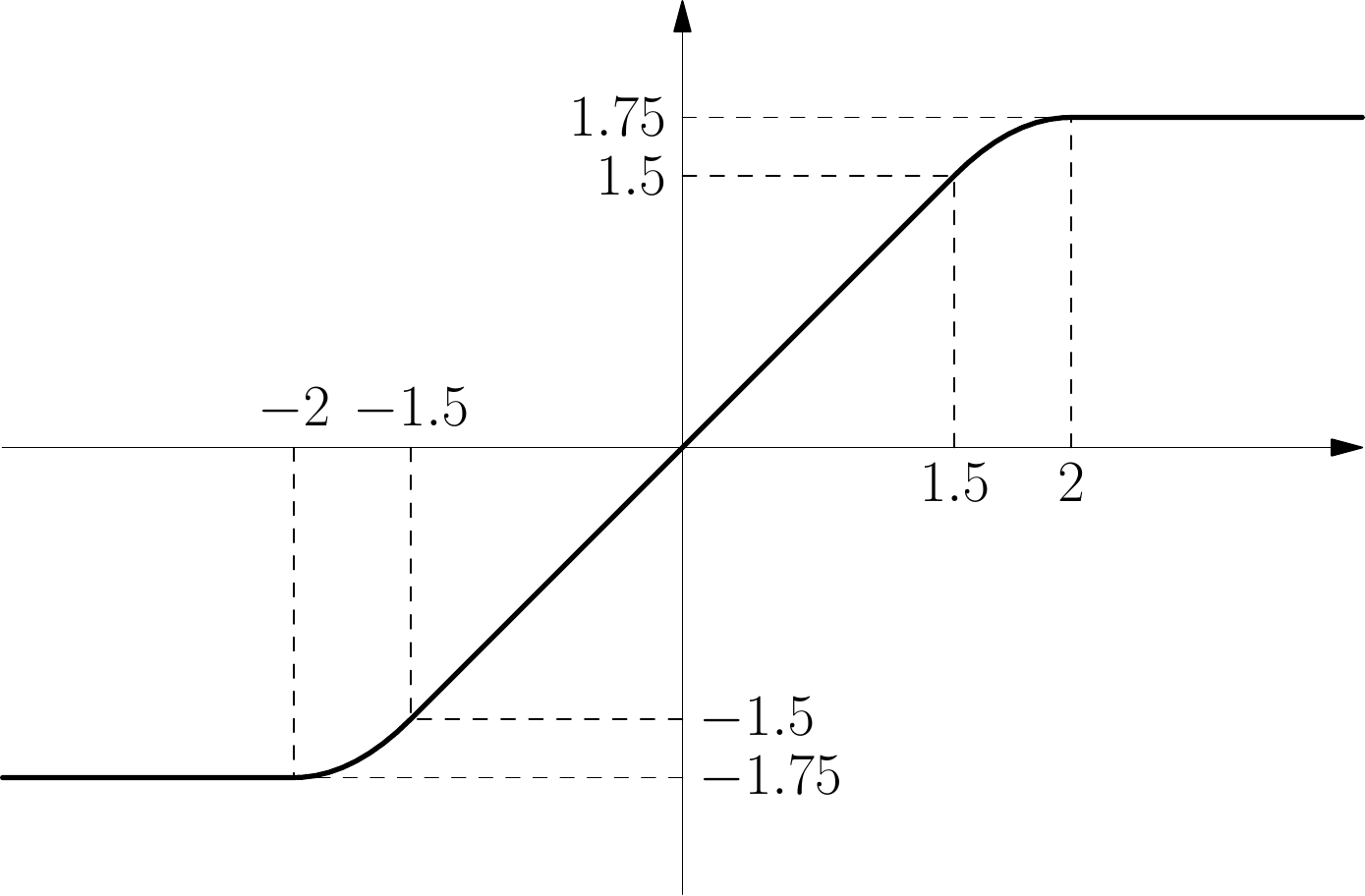}
   \caption{The graph of $\Theta$} \label{Theta2}
\end{figure}
 \pagebreak
 The function $\Theta$ has the following properties.
 \begin{itemize}
 \item $\Theta(x)=x$ for $x\in [-1.5,1.5]$,
 \item $ \Theta $ is $C^\infty$-smooth and $|\Theta| \leq 1.75 $,
 \item $|\Theta'|\leq 1$ and $|\Theta(x)| \leq |x|$.
 \end{itemize}

 Put $u_k = \Theta(\tilde u_k)$. Note that $u_k= \tilde u_k$ and $\nabla u_k = \nabla \tilde u_k$ if $|\tilde u_k|\leq 1.5$ and we always have $|u_k|\leq  |\tilde u_k|$, $|\nabla u_k|\leq  |\nabla \tilde u_k|$. 
 We need to show that $\|u_k-u\|_2$, $\|\nabla u_k- \nabla u\|_2 \to 0$.
 Since $\tilde u_k$ converge in $W^{1,2}_0(\Omega)$, the functions  $|\tilde u_k|^2$ and $|\nabla \tilde u_k|^2$ are uniformly integrable, 
 i.e., for any $\varepsilon>0$, there is $\delta>0$ such that if $m_2(E)<\delta$, then 
 $$\int_E |\tilde u_k|^2, \int_E |\nabla \tilde u_k|^2 < 
 \varepsilon.$$
 In the following computation $\int$ will denote the integral over $\Omega$:
 $$ \int |u_k-u|^2 \leq 2 \left[ \int |u_k-\tilde u_k|^2 +\int |\tilde u_k - u|^2 \right]\leq$$
 
  $$ \leq   4 \!\!\!\!\!\!\int\limits_{\{|\tilde u_k|>1.5\}} (|u_k|^2+|\tilde u_k|^2) + 2\int |\tilde u_k - u|^2 \leq$$
  
  $$\leq  8\!\!\!\!\!\!\int\limits_{\{|\tilde u_k|>1.5\}} |\tilde u_k|^2 + 2 \int|\tilde u_k - u|^2.$$ 
  The second term tends to zero by the choice of $\tilde u_k$.
  Note that $$m_2(\{|\tilde u_k|>1.5\}) \leq m_2(|u-\tilde u_k|\geq 0.5) \leq 4 \|u-\tilde u_k\|_2^2 \to 0.$$ So the first term tends to zero by the uniform integrability property. In a similar way one can show that
  $\int|\nabla u_k - \nabla u|^2 \to 0 $.

  Now we would like to choose a subsequence such that
  $ u_k \to u, \nabla u_k \to \nabla u $  almost everywhere in $\Omega$.
  It can be done by choosing any subsequence $u_{k_j}$ with $$\|u_{k_j}-u\|_2 \leq \frac{1}{4^j}, \quad
  \|\nabla u_{k_j}- \nabla u\|_2 \leq \frac{1}{4^j}.$$
  Let $E_j$ be the set of points $x\in \Omega$ such that $|u_{k_j}(x)-u(x)| \geq \frac{1}{2^{j}}$.
  Then $$  \sqrt{m_2(E_j) \frac{1}{4^{j}} } \leq \|u_{k_j}-u\|_2 \leq \frac{1}{4^{j}}, \quad {\text{so}} \quad  m_2(E_j) \leq \frac{1}{4^j}.$$ Note that if $x \notin \cup_{j=n}^\infty E_j$, then $u_{k_j}(x)$ converge to $u(x)$.
  However  $$m_2(\cup_{j=n}^\infty E_j) \leq \sum_{j=n}^\infty m_2( E_j) \leq \frac{1}{2^{n}}.$$
  Thus  $u_{k_j}$ converge to $u$ almost everywhere in $\Omega$. In a similar way one can show that $\nabla u_{k_j}$ also converge to $\nabla u$ almost everywhere.

  \end{proof}
  
  \noindent \textbf{Fact \ref{product}.}
  Let $\Omega$ be an open set. Assume that $u,v \in W^{1,2}_{loc}(
\Omega)\cap L^{{}^{\scriptsize \infty}}_{loc}(\Omega)$. then $uv \in W^{1,2}_{loc}(
\Omega)  $ and $\nabla(uv) = u \nabla  v + v \nabla u$.

\begin{proof}
Clearly, the fact is local. So we may assume $\Omega= B(0,r)$ and $u,v \in W^{1,2}(B(0,r)
)\cap L^\infty(B(0,r))$.

Let us fix a small $\delta>0$ and
let $$K_\varepsilon(z)=\frac{1}{\varepsilon^2}\kappa(|z|/\varepsilon)$$ be a $C^\infty$-approximation to identity with $\text{supp} K_\varepsilon(z)\subset B(0,\varepsilon)$, $\varepsilon < \delta$.
Then $K_\varepsilon* u$ and $K_\varepsilon*v$
converge to $u$ and $v$ in $L^2(B(0,r-\delta))$ and a.e. in $B(0,r-\delta)$ as $\varepsilon \to 0$. Consider any test function $\eta \in C_0^\infty(B(0,r-\delta))$ and extend it by $0$ outside $B(0,r-\delta)$. Then $\eta*K_\varepsilon \in C_0^\infty(B(0,r))$ and $K_\varepsilon * \nabla \eta = \nabla (K_\varepsilon*\eta) $.
By Fubini's theorem  we have
$$ \int_{B(0,r-\delta)} (K_\varepsilon* u) \nabla \eta =\int_{B(0,r)} u (K_\varepsilon*  \nabla \eta)=$$
$$=\int_{B(0,r)} u \nabla(K_\varepsilon*  \eta)= -\int_{B(0,r)} \nabla u (K_\varepsilon*  \eta)= $$
$$=- \int_{B(0,r)}\left(\int_{B(0,r-\delta)} \nabla u(x) K_\varepsilon(x-y)\eta(y) dy \right) dx=$$ $$= - \int_{B(0,r-\delta)}\left(\int_{B(0,r)} \nabla u(x) K_\varepsilon(y-x)\eta(y) dx \right) dy = -\int_{B(0,r-\delta)} (\nabla u *K_\varepsilon)  \eta.$$

So $$u_\varepsilon:=K_\varepsilon* u \quad \textup{and} \quad v_\varepsilon:=\quad K_\varepsilon* v$$ are in $W^{1,2}(B(0,r-\delta))\cap C^\infty(B(0,r-\delta))$ with 
\begin{enumerate}

\item $\nabla u_\varepsilon =  \nabla u * K_\varepsilon$,$\nabla v_\varepsilon =  \nabla v * K_\varepsilon$, 
\item
$\nabla u * K_\varepsilon \to \nabla u$, $\nabla v * K_\varepsilon \to \nabla v$ in $L^2(B(0,r-\delta))$ and a.e. in $B(0,r-\delta)$ as $\varepsilon \to 0$,
\item By Young's inequality for convolutions, we have $|u_\varepsilon|< \|u\|_{L^\infty(B(0,r))}$, $|v_\varepsilon|< \|v\|_{L^\infty(B(0,r))}$ in $B(0,r-\delta)$.
\end{enumerate}

We know that the convergence of $u_\varepsilon v_\varepsilon$ holds  a.e. in $B(0,r-\delta)$ and
$|u_\varepsilon v_\varepsilon|$ are bounded by $\|u\|_\infty\|v\|_\infty$ a.e. in $B(0,r-\delta)$ if $\varepsilon< \delta$. So by the Lebesgue dominated convergence theorem $u_\varepsilon v_\varepsilon \to uv$  in  $L^2(B(0,r-\delta))$.

We want to show that 
$$\nabla(uv) = u \nabla  v + v \nabla u$$
in the sense of $W^{1,2}(B(0,r-\delta))$. It is clear that $u \nabla  v + v \nabla u$ is in $L^2(B(0,r-\delta))$ because $u,v$ are bounded and their gradients are in $L^2(B(0,r-\delta))$.

Consider again  a test function $\eta \in C_0^\infty(B(0,r-\delta))$. We have
 $$ \int uv \nabla \eta = \lim_{\varepsilon \to 0} \int u_\varepsilon v_\varepsilon \nabla \eta = - \lim_{\varepsilon \to 0}  \int (\nabla u_\varepsilon v_\varepsilon + u_\varepsilon \nabla v_\varepsilon ) \eta=$$
 $$ =- \lim_{\varepsilon \to 0} \left[ \int \left((\nabla u_\varepsilon - \nabla u) v_\varepsilon + u_\varepsilon (\nabla v_\varepsilon - \nabla v) \right) \eta + \int ( \nabla u v_\varepsilon + u_\varepsilon \nabla  v ) \eta  \right].$$
 Note that $\int (\nabla u_\varepsilon - \nabla u) v_\varepsilon \eta  \to 0$ because $\nabla u_\varepsilon\to  \nabla u$ in $L^2(B(0,r-\delta))$ and  $|v_\varepsilon \eta | < \|v\|_\infty\| \eta\|_\infty$ in  $B(0,r-\delta)$. Similarly, $ \int u_\varepsilon (\nabla v_\varepsilon - \nabla v) \eta \to 0$.
Finally, by the Lebesgue dominated convergence theorem $$ \int ( \nabla u v_\varepsilon + u_\varepsilon \nabla  v ) \eta  \to \int ( \nabla u v + u \nabla  v ) \eta$$
because the convergence of the functions holds a.e. in  $B(0,r-\delta)$ and there is the integrable majorant
$(|\nabla u| \|v\|_\infty + \|u\|_\infty |\nabla  v |) \|\eta\|_\infty$.
Thus 
$\int uv \nabla \eta = - \int ( \nabla u v + u \nabla  v ) \eta$.

\end{proof}

\begin{fact} \label{fact6}
Let $\Omega$ be a bounded open set. Let $\varphi=1+\psi$, where $\psi \in W^{1,2}_0(\Omega)$,
$\|\psi\|_\infty \leq \frac{1}{3}$. Then the functions
$$ \tilde \varphi =
 \begin{cases}
\varphi \textup{ in } \Omega \\
1 \textup{ outside } \Omega
\end{cases}
\quad \textup{ and } \quad
\eta=
\begin{cases}
\frac 1 \varphi \textup{ in } \Omega \\
1 \textup{ outside } \Omega
\end{cases}
$$
are in $W^{1,2}_{loc}(\mathbb{R}^2)$ and 
$$\nabla \tilde \varphi= \nabla \varphi \mathbbm{1}_\Omega, \quad \nabla \eta = - \frac{\nabla \varphi}{\varphi^2}\mathbbm{1}_\Omega.$$
\end{fact}
\begin{proof}
Consider a sequence of functions $\psi_k \in C_0^\infty(\Omega)$ such that
$\| \psi_k\|_\infty  \leq \frac 2 3 $ and $\psi_k \to \psi$, $\nabla \psi_k \to \nabla \psi$ in $L^2(\Omega)$
and a.e. in $\Omega$. We can extend $\psi_k$ by zero outside $\Omega$ and get a sequence of $C_0^\infty(\mathbb{R}^2)$ functions, which we will still denote by $\psi_k$, such that $
\psi_k=0, \nabla \psi_k=0$  in $\mathbb{R}^2\setminus \Omega$ while $\psi_k \to \psi \mathbbm{1}_\Omega$ in $L^2(\mathbb{R}^2)$ and a.e., $\nabla \psi_k \to \nabla \psi \mathbbm{1}_\Omega$ in $L^2(\mathbb{R}^2)$. This immediately implies that $1+\psi_k \to \tilde \varphi$, $\nabla (1+\psi_k)=\nabla \psi_k \to \nabla \psi \mathbbm{1}_\Omega$ in $L^2_{loc}(\mathbb{R}^2)$,
so $\tilde \varphi \in W^{1,2}_{loc}(\mathbb{R}^2)$ and $\nabla \tilde \varphi= \tilde \varphi \mathbbm{1}_\Omega.$
Note that 
$$ \left| \frac{1}{1+\psi_k}- \frac{1}{1+\psi}\right|= \left| \frac{\psi_k-\psi}{(1+\psi_k)(1+\psi)}\right| \leq 9|\psi_k - \psi|$$
and 
$$ \left| \nabla \frac{1}{1+\psi_k}+ \frac{\nabla \psi}{(1+\psi)^2}\right|=  \left| \frac{\nabla \psi_k}{(1+\psi_k)^2} - \frac{\nabla \psi}{(1+\psi)^2}\right|\leq$$
$$\leq |\nabla \psi_k -\nabla \psi| \frac{1}{(1+\psi_k)^2}+|\nabla \psi|\left| \frac{1}{(1+\psi_k)^2} - \frac{1}{(1+\psi)^2}\right| \quad \textup{ in } \Omega.$$
Also $\frac{1}{1+\psi_k}=1$, $\nabla\frac{1}{1+\psi_k}=0$ in $\mathbb{R}^2\setminus \Omega$.
Hence $\frac{1}{1+\psi_k} \to \eta$ in $L^2(\mathbb{R}^2)$ and we would like to show that
$$\nabla\frac{1}{1+\psi_k} \to - \frac{\nabla \psi}{(1+\psi)^2} \mathbbm{1}_\Omega= - \frac{\nabla \varphi}{\varphi^2}\mathbbm{1}_\Omega \quad \textup{ in } L^2(\mathbb{R}^2).$$
To see the latter, note that $\frac{1}{(1+\psi_k)^2} \leq 9$, so
$$\int_\Omega |\nabla \psi_k - \nabla \psi|^2 \frac{1}{(1+\psi_k)^4}  \leq 81 \int_\Omega |\nabla \psi_k - \nabla \psi|^2 \to 0,$$
and the functions $|\nabla \psi|^2 \left[ \frac{1}{(1+\psi_k)^2} - \frac{1}{(1+\psi)^2} \right]^2$
have the integrable majorant $81|\nabla \psi|^2$ and tend to $0$ almost everywhere in $\Omega$.
Thus $\eta \in W^{1,2}_{loc}(\mathbb{R}^2)$ and $\nabla \eta = - \frac{\nabla \varphi}{\varphi^2}\mathbbm{1}_\Omega$
as required.

\end{proof}

\begin{lemma} \label{le: W0} Let $\Omega$ be a bounded open set and let a function $f\in C^1(\overline{\Omega})$ be zero on $\partial \Omega$. Then $f \in W^{1,2}_{0}(\Omega)$.
\end{lemma}

\begin{proof}
Let $\varepsilon>0$, denote by $\Omega_\varepsilon$ the set of points $x$ in $\Omega$ with distance to the boundary of $\Omega$ at least $\varepsilon$.  Let $\eta$ be a function in $C_0^\infty(\Omega)$ with the following properties:
\begin{itemize}
\item $\eta(x)= 1$, if $x \in \Omega_\varepsilon$.
\item $0 \leq \eta \leq 1$ and $|\nabla \eta| \leq \frac{C}{\varepsilon}$ in $\Omega$.
\end{itemize}
The function $f\eta$ is in $C_0^1(\Omega) \subset W_0^{1,2}(\Omega) $.
We want to show that $f\eta$ converge to $f$ in $W_0^{1,2}(\Omega)$ norm as $\varepsilon \to 0$. Observe that $|\nabla f|$ is uniformly bounded in $\Omega$ by some constant $A=A(f)$, so $|f(x)| \leq A\varepsilon $ if the distance from $x$ to $\partial \Omega$ is smaller than $\varepsilon$ and therefore $$|\nabla(f\eta)|\leq |\nabla f||\eta|+  | f|| \nabla \eta| \leq A + AC \quad \text{ in } \Omega\setminus \Omega_\varepsilon.$$
Then $$\int\limits_\Omega |f -f \eta|^2 = \int\limits_{\Omega \setminus \Omega_\varepsilon } |f -f \eta|^2  \leq \int\limits_{\Omega \setminus \Omega_\varepsilon } |f|^2 \to 0$$
and 
$$\int\limits_\Omega |\nabla f - \nabla (f \eta)|^2 = \int\limits_{\Omega \setminus \Omega_\varepsilon } |\nabla f -\nabla(f \eta)|^2 \leq 2 \int\limits_{\Omega \setminus \Omega_\varepsilon } (|\nabla f|^2 +|\nabla(f \eta)|^2) \leq$$
$$  \leq m_2(\Omega \setminus \Omega_\varepsilon) C_1 A^2. $$
Since  $m_2(\Omega \setminus \Omega_\varepsilon) \to 0$ as $\varepsilon \to 0$, we have verified that $f \in W_0^{1,2}(\Omega)$. 
\end{proof}
\begin{lemma} \label{le: diameter}
Let $u$ be a solution to $\Delta u + Vu=0$, $|V|\leq 1$, in a ball $B(x,r)$,
where $r<r_0$ and $r_0$ is a sufficiently small universal constant.
If $u$ is continuous up to $\partial B(x,r)$ and $u>0$ on $\partial B(x,r)$, then
$u>0$ in $B(x,r)$.
\end{lemma}
\begin{proof}
We may assume that $u$ is larger than a positive constant $\delta$ on $\partial B(x,r)$.
Consider the set $\Omega=\{ x \in B(x,r): u(x)< \frac \delta 2 \}$.
This is an open set strictly inside $B(x,r)$ and if $u$ is not positive in $B(x,r)$, then $\Omega$  is not empty.

Since $u \in C^1(\overline{\Omega})$ by Fact \ref{fact5} and $u=\frac \delta 2$ on $\partial \Omega$, we know by Lemma \ref{le: W0} that $(u-\frac \delta 2) \in W_0^{1,2}(\Omega)$. 

Note that $\overline{\Omega} \subset B(0,r) $, so $\Omega$ has a Poincare constant smaller than $Cr^2$.
By Lemma \ref{solving Schrodinger}, if $r$ is sufficiently small, we can find $\varphi=1+\tilde \varphi$ with
$\tilde \varphi \in W^{1,2}_0(\Omega), \|\tilde \varphi\|_\infty<\frac{1}{2}$ 
such that $\varphi$ is a solution to $\Delta \varphi + V \varphi=0$ in $\Omega$.
Then $ (\frac{\delta}{2}\varphi - \frac{\delta}{2}) \in W^{1,2}_0(\Omega)$ and therefore the function
$g=(\frac{\delta}{2}\varphi - u) \in W^{1,2}_0(\Omega) $. The function $g$ is also a solution to $\Delta g+ Vg=0$. For any $\eta\in C_0^\infty(\Omega)$, we have $\int_\Omega \nabla g \nabla \eta=  \int_\Omega Vg\eta$
and taking the limit in $W^{1,2}_0(\Omega)$, we get
$$ \int_\Omega |\nabla g|^2 = \int_\Omega Vg^2 \leq \int_\Omega g^2.$$
However Poincare's inequality implies 
$$ \int_\Omega g^2 \leq Cr^2 \int_\Omega |\nabla g|^2. $$
If $r$ is sufficiently small, this could happen only if $g=0$ in $\Omega$.
So $u= \frac{\delta}{2}\varphi$  in $\Omega$, but $\varphi> \frac 1 2$ in $\Omega$.
So $u>\frac \delta 4$ in $\Omega$ and in $B(x,r)$.
\end{proof}

\subsection{Divergence free vector fields on the plane} \label{sec: divergence}
If $F=(F_1,F_2): B \to \mathbb{R}^2$ is a $C^1$- smooth vector field in a disk $B$ on the plane such
that $F$ is divergence free: $\text{div} F=0$ in $B$, then 
there is a smooth function $u$ such that $$(F_1,F_2)= \nabla \times u := (u_{x_2},-u_{x_1}).$$

Sometimes people refer to the statement above as to Poincare's lemma or the fundamental theorem of calculus, or the inverse gradient theorem. Here is the sketch of the standard proof. WLOG, $B=B(0,1)$. Consider any point $Q\in B$ and
the rectangle $R \subset B$ with opposite vertices $0$ and $Q$, and sides parallel to $x_1$ and $x_2$ axes.

  \begin{figure}[h!] 
    \includegraphics[width=0.5\textwidth]{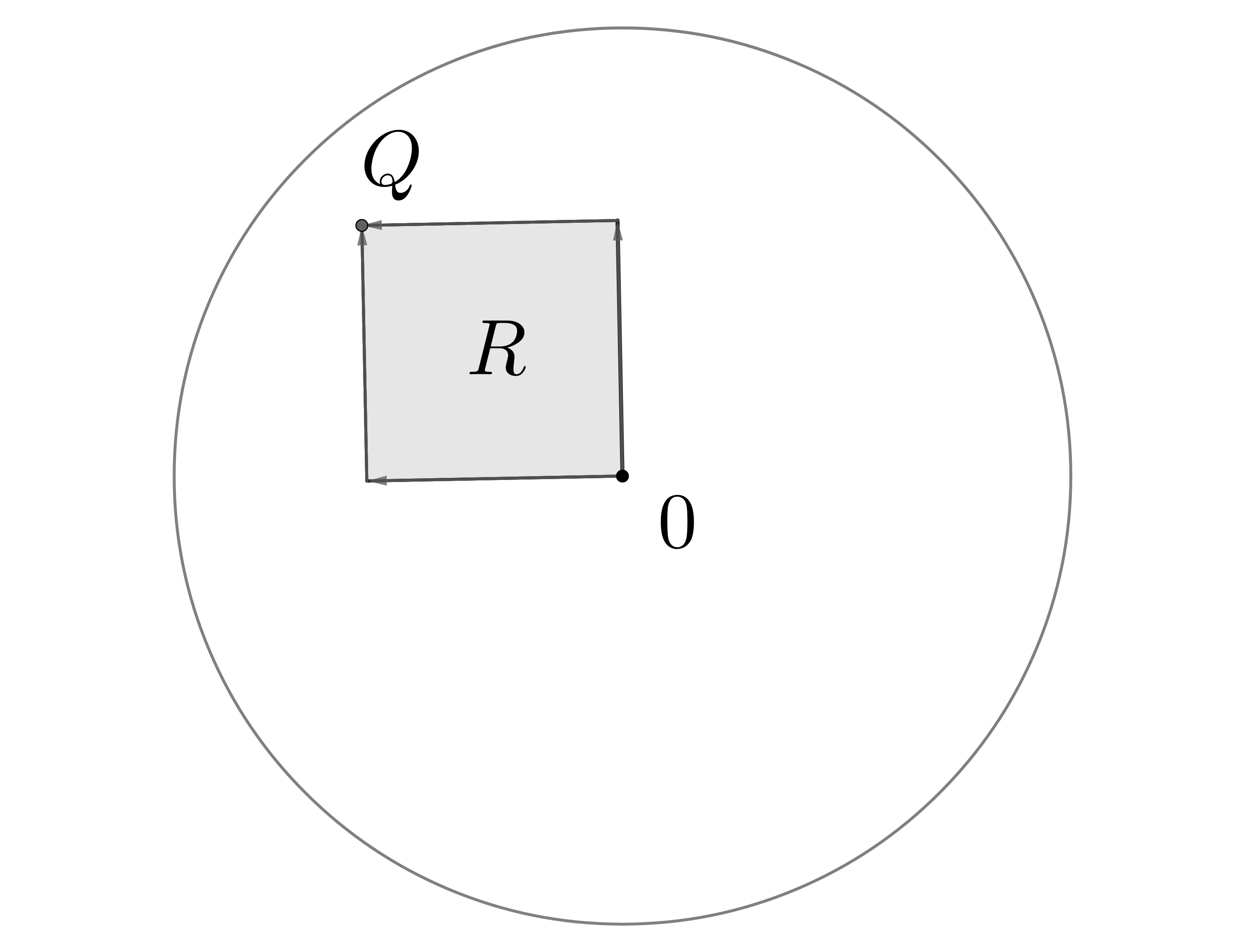}
   
\end{figure}

 Note that
that the contour integral  $$  \int_{\partial R} (-F_2,F_1)\cdot dx= \int_{\partial R} F\cdot n(x) |dx| = \int_R \text{div } F$$
is zero. There are two simple paths that start at $0$, go along the sides of $R$ and end at $Q$. The integrals $\int(-F_2,F_1)dx$ over those two paths are the same and we define $u(Q)$ to be equal to both of them. The  differentiation of $u(Q)$ in the horizontal and vertical directions shows
that $(F_1,F_2)= (u_{x_2},-u_{x_1})$.

We need the version with less regularity assumptions on the divergence free vector field $F$:
if $F\in L_{\text{loc}}^p(B(0,1))$, $1\leq p<\infty$, and $\int_{B(0,1)}F\nabla h=0$ for any $h \in C^\infty_0(B(0,1))$ (the divergence free condition), then there is a function $u\in W^{1,p}_{loc}(B(0,1))$  such that $(F_1,F_2)= \nabla \times u := (u_{x_2},-u_{x_1}).$

Indeed, let $$K_\varepsilon(z)=\frac{1}{\varepsilon^2}\kappa(|z|/\varepsilon)$$ be a $C^\infty$-approximation to identity with $\text{supp} K_\varepsilon(z)\subset B(0,\varepsilon)$.
Define $$F_\varepsilon= F*K_\varepsilon= (F_1*K_\varepsilon, F_2*K_\varepsilon)$$ in the smaller ball $B(0,1-\varepsilon)$. Then 
$F_\varepsilon$ is divergence free in $B(0,1-\varepsilon)$.
Indeed, if $f\in C_0^\infty(B(0,1-\varepsilon)$, then by Fubini's theorem
 $$\int (F*K_\varepsilon) \nabla f= \int F (K_\varepsilon * \nabla f)=\int F \nabla (K_\varepsilon * f)=0.$$
 So there is a $C^\infty$ function $u_\varepsilon$ such that $F_\varepsilon=\nabla \times u_\varepsilon$ in $B(0,1-\varepsilon)$. Fix $\delta \in (0,1)$.
 By the Lebesgue theory $F*K_\varepsilon$ converge to $F$ in $L^p(B(0,1-\delta))$.
 Thus $\nabla u_\varepsilon$ is a Cauchy sequence in $L^p(B(0,1-\delta))$. Let us add a constant to $u_\varepsilon$ so that $\int_{B(0,1-\delta)}u_\varepsilon=0$. By the Poincaré--Wirtinger inequality (see p.275, Theorem 1 in \cite{P}) $u_\varepsilon$ is a Cauchy sequence in $L^p(B(0,1-\delta))$.
 Thus we can find a function $\tilde u_\delta$ such that $u_\varepsilon$ converge to $\tilde u_\delta$ in $W^{1,p}(B(0,1-\delta))$
 and $\nabla \times \tilde u_\delta = \lim_{\varepsilon \to 0} \nabla \times u_\varepsilon=(F_1,F_2)$.
 For any $\delta_1,\delta_2 \in (0,1)$, the gradients of $\tilde u_{\delta_1}$ and  $\tilde u_{\delta_2}$ are 
 the same in $B(0,1-\max(\delta_1,\delta_2))$ and therefore $\tilde u_{\delta_1} - \tilde u_{\delta_2}$ is constant almost everywhere in $B(0,1-\max(\delta_1,\delta_2))$.
 Finally, let us modify $\tilde u_\delta $ by subtracting a constant  so that $ \int_{B(0,1/2)} \tilde u_\delta = 0 $
 for all $\delta <1/2$. Then $u$ is well-defined by $u=\tilde u_\delta$ in $B(0,1-\delta)$.

\end{document}